\documentclass[11pt,reqno]{amsart}
\usepackage{fullpage}
\allowdisplaybreaks
\renewcommand{\le}{\leqslant}
\renewcommand{\ge}{\geqslant}

\renewcommand{\setminus}{\smallsetminus}

\renewcommand{\subset}{\subseteq}

\newcommand{\n}{\{1,\ldots,n\}}

\newcommand{\B}{\mathfrak{B}}
\newcommand{\e}{\varepsilon}
\usepackage{amsmath,amsthm,amsfonts,amssymb,verbatim}
\usepackage{paralist}
\usepackage{enumerate}
\usepackage[mathscr]{euscript}
\usepackage{color}
\usepackage{upgreek}
\newcommand{\W}{\mathsf{W}}
\renewcommand{\gamma}{\upgamma}
\newcommand{\1}{\mathbf 1}
\usepackage{graphicx,verbatim}
\usepackage{paralist}
\usepackage[breaklinks,pdfstartview=FitH]{hyperref}

\newcommand{\alert}[1]{\textbf{\color{red}
[[[#1]]]}\marginpar{\textbf{\color{red}**}}\typeout{ALERT:
\the\inputlineno: #1}}
\DeclareMathOperator{\dist}{\bf dist}
\newcommand{\Pp}{\mathscr{P}_{\!p}}

\newcommand{\diam}{{\rm diam}}

\newcommand{\N}{\mathbb{N}}

\newcommand{\R}{\mathbb{R}}

\newcommand{\E}{{\mathbb{E}}}

\newcommand{\mommit}[1]{}
\newcommand{\namedref}[2]{\hyperref[#2]{#1~\ref*{#2}}}

\theoremstyle{plain}
\newtheorem{theorem}{Theorem}
\newtheorem{lemma}[theorem]{Lemma}
\newtheorem{corollary}[theorem]{Corollary}
\newtheorem{remark}[theorem]{Remark}

\newtheorem{proposition}[theorem]{Proposition}

\theoremstyle{definition}

\newtheorem{question}[theorem]{Question}

\newenvironment{RETHM}[2]{\trivlist \item[\hskip 10pt\hskip\labelsep{\bf
#1\hskip 5pt\relax\ref{#2}.}]\it}{\endtrivlist}
\newcommand{\rethm}[1]{\begin{RETHM}{Theorem}{#1}}
\newcommand{\repro}[1]{\begin{RETHM}{Proposition}{#1}}
\newcommand{\relem}[1]{\begin{RETHM}{Lemma}{#1}}
\newcommand{\recor}[1]{\begin{RETHM}{Corollary}{#1}}
\newcommand{\erethm}{\end{RETHM}}
\renewcommand{\epsilon}{\varepsilon}

\renewcommand{\d}{\mathrm{d}}
\newcommand{\eqdef}{\stackrel{\mathrm{def}}{=}}

\title[Snowflake universality of Wasserstein spaces]{Snowflake universality  of Wasserstein spaces}

\author{Alexandr Andoni}
\address{Computer Science Department, Columbia University, 500 West 120th Street, Room 450, MC0401, New York, NY 10027, USA.}
\email{andoni@cs.columbia.edu}
\author{Assaf Naor}
\address{Mathematics Department\\ Princeton University\\ Fine Hall, Washington Road, Princeton, NJ 08544-1000, USA.}
\email{naor@math.princeton.edu}
\thanks{A.~N. was supported in part by the BSF, the Packard Foundation and the Simons Foundation.}
\author{Ofer Neiman}
\address{Department of Computer Science, Ben-Gurion University of the Negev, P.O.B 653 Be'er Sheva 84105, Israel}
\email{neimano@cs.bgu.ac.il}
\thanks{O.~N. was supported in part by the ISF and the European Union's Seventh
Framework Programme}

\begin{document}
\maketitle

\begin{abstract}
For $p\in (1,\infty)$ let $\Pp(\R^3)$ denote the metric space of all $p$-integrable Borel probability measures on $\R^3$, equipped with the Wasserstein $p$ metric $\W_p$. We prove that for every $\e>0$, every $\theta\in (0,1/p]$ and every finite metric space $(X,d_X)$, the metric space $(X,d_{X}^{\theta})$ embeds into $\Pp(\R^3)$ with distortion at most $1+\e$. We show that this is sharp when $p\in (1,2]$ in the sense that the exponent $1/p$ cannot be replaced by any larger number. In fact, for arbitrarily large $n\in \N$ there exists an $n$-point metric space $(X_n,d_n)$ such that for every $\alpha\in (1/p,1]$ any embedding of the metric space $(X_n,d_n^\alpha)$ into $\Pp(\R^3)$ incurs distortion that is at least a constant multiple of $(\log n)^{\alpha-1/p}$. These statements establish that there exists an Alexandrov space of nonnegative curvature, namely $\mathscr{P}_{\! 2}(\R^3)$, with respect to which there does not exist a sequence of bounded degree expander graphs. It also follows that $\mathscr{P}_{\! 2}(\R^3)$ does not admit a uniform, coarse, or quasisymmetric embedding into  any Banach space of nontrivial type. Links to several longstanding open questions in metric geometry are discussed, including the characterization of subsets of Alexandrov spaces, existence of expanders, the universality problem for $\mathscr{P}_{\! 2}(\R^k)$, and the metric cotype dichotomy problem.
\end{abstract}


\section{Introduction}

We shall start by quickly recalling basic notation and terminology from the theory of transportation cost metrics; all the necessary background can be found in~\cite{Vil03}. For a complete separable metric space $(X,d_X)$ and $p\in (0,\infty)$,
let $\Pp(X)$ denote the space  of all Borel probability measures
$\mu$ on $X$ satisfying $$\int_X d_X(x,x_0)^p\d\mu(x)<\infty$$ for some
(hence all) $x_0\in X$. A coupling of a pair of Borel probability measures $(\mu,\nu)$ on $X$ is a Borel probability measure $\pi$ on $X\times X$ such that $\mu(A)=\pi(A\times X)$ and $\nu(A)=\pi(X\times A)$ for every Borel measurable $A\subset X$. The set of couplings of $(\mu,\nu)$ is denoted $\Pi(\mu,\nu)$. The Wasserstein $p$ distance between $\mu,\nu\in \Pp(X)$ is defined to be
$$
\W_p(\mu,\nu)\eqdef \inf_{\pi\in \Pi(\mu,\nu)}\bigg(\iint_{X\times X} d_X(x,y)^p\d\pi(x,y)\bigg)^{\frac{1}{p}}.
$$
$\W_p$ is a metric on $\Pp(x)$ whenever $p\ge 1$. The metric space $(\Pp(X),\W_p)$ is called the Wasserstein $p$ space over $(X,d_X)$. Unless stated otherwise, in the ensuing discussion whenever we refer to the metric space $\Pp(X)$ it will be understood that $\Pp(X)$ is equipped with the metric $\W_p$.

\subsection{Bi-Lipschitz Embeddings} Suppose that $(X,d_X)$ and $(Y,d_Y)$ are metric spaces and that $D\in [1,\infty]$. A mapping $f:X\to Y$ is said to have distortion at most $D$ if there exists $s\in (0,\infty)$ such that every $x,y\in X$ satisfy $sd_X(x,y)\le d_Y(f(x),f(y))\le Dsd_X(x,y)$. The infimum over those $D\in [1,\infty]$ for which this holds true is called the distortion of $f$ and is denoted $\dist(f)$. If there exists a mapping $f:X\to Y$ with distortion at most $D$ then we say that $(X,d_X)$ embeds with distortion $D$ into $(Y,d_Y)$. The infimum of $\dist(f)$ over all $f:X\to Y$ is denoted $c_{(Y,d_Y)}(X,d_X)$, or $c_Y(X)$ if the metrics are clear from the context.

\subsection{Snowflake universality} Below, unless stated otherwise, $\R^n$ will be endowed with the standard Euclidean metric. Here we show that $\Pp(\R^3)$ exhibits the following universality phenomenon.

\begin{theorem}\label{thm:main snowflake}
If $p\in (1,\infty)$ then for every finite metric space $(X,d_X)$ we have
$$
c_{\left(\Pp(\R^3),\W_p\right)}\Big(X,d_X^{\frac{1}{p}}\Big)=1.
$$
\end{theorem}
For a metric space $(X,d_X)$ and $\theta\in (0,1]$, the metric space $(X,d_X^\theta)$ is commonly called the $\theta$-{\em snowflake} of $(X,d_X)$; see e.g.~\cite{DS97}. Thus Theorem~\ref{thm:main snowflake} asserts that the $\theta$-snowflake of any finite metric space $(X,d_X)$ embeds with distortion $1+\e$ into $\Pp(\R^3)$ for every $\e\in (0,\infty)$ and $\theta\in (0,1/p]$ (formally, Theorem~\ref{thm:main snowflake}  makes this assertion when $\theta=1/p$, but for general  $\theta\in (0,1/p]$  one can then apply Theorem~\ref{thm:main snowflake} to the metric space $(X,d_X^{\theta p})$ to deduce the seemingly more general statement).

Theorem~\ref{thm:no large snowflake}  below implies that Theorem~\ref{thm:main snowflake} is sharp   if $p\in (1,2]$, and yields a nontrivial, though probably non-sharp, restriction on the embeddability of snowflakes into $\Pp(\R^3)$ also for $p\in (2,\infty)$.

\begin{theorem}\label{thm:no large snowflake} For arbitrarily large $n\in \N$ there exists an $n$-point metric space $(X_n,d_{X_n})$ such that for every $\alpha \in (0,1]$ we have
\begin{equation*}
c_{(\Pp(\R^3),\W_p)}(X_n,d_{X_n}^\alpha)\gtrsim \left\{\begin{array}{ll}(\log n)^{\alpha-\frac{1}{p}}&\mathrm{if}\ p\in (1,2],\\
(\log n)^{\alpha+\frac{1}{p}-1}&\mathrm{if}\ p\in (2,\infty).\end{array}\right.
\end{equation*}
\end{theorem}
Here, and in what follows, we use standard asymptotic notation, i.e., for $a,b\in [0,\infty)$ the notation $a\gtrsim b$ (respectively $a\lesssim b$) stands for $a\ge c b$ (respectively $a\le cb$) for some universal constant $c\in (0,\infty)$. The notation $a\asymp b$ stands for $(a\lesssim b)\wedge (b\lesssim a)$. If we need to allow the implicit constant to depend on parameters we indicate this by subscripts, thus $a\lesssim_p b$ stands for $a\le c_p b$ where $c_p$ is allowed to depend only on $p$, and similarly for the notations $\gtrsim_p$ and $\asymp_p$.

We conjecture that when $p\in (2,\infty)$ the lower bound in Theorem~\eqref{thm:no large snowflake}  could be improved to
$$
c_{(\Pp(\R^3),\W_p)}(X_n,d_{X_n}^\alpha)\gtrsim_p (\log n)^{\alpha-\frac12},
$$
and, correspondingly, that the conclusion of Theorem~\ref{thm:main snowflake} could be improved to state that if $p\in (2,\infty)$ then  $c_{\left(\Pp(\R^3),\W_p\right)}\!\left(X,\sqrt{d_X}\right)\lesssim_p 1$ for every finite metric space $(X,d_X)$; see Question~\ref{Q:p>2} below.

There are several motivations for our investigations that led to Theorem~\ref{thm:main snowflake} and  Theorem~\ref{thm:no large snowflake}. Notably, we are inspired by a longstanding open question of Bourgain~\cite{Bou86}, as well as fundamental questions on the geometry of Alexandrov spaces. We shall now explain these links.

\subsection{Alexandrov geometry} We need to briefly present some standard  background on metric spaces that are either nonnegatively curved or nonpositively curved in the sense of Alexandrov; the relevant background  can be found in e.g.~\cite{BGP92,BH99}. Let $(X,d_X)$ be a complete geodesic metric space. Recall that $w\in X$ is called a metric midpoint of  $x,y\in X$ if $d_X(x,w)=d_X(y,w)=d_X(x,y)/2$.  The metric space $(X,d_X)$ is said to be an Alexandrov space of nonnegative curvature if for every $x,y,z\in X$ and every metric midpoint $w$ of $x,y$,
\begin{equation}\label{eq:def alexandrov nonnegative}
d_X(x,y)^2+4d_X(z,w)^2\ge 2d_X(x,z)^2+2d_X(y,z)^2.
\end{equation}
Correspondingly, the metric space $(X,d_X)$ is said to be an Alexandrov space of nonpositive curvature, or a Hadamard space, if for every $x,y,z\in X$ and every metric midpoint $w$ of $x,y$,
\begin{equation}\label{eq:def hadamrd}
d_X(x,y)^2+4d_X(z,w)^2\le 2d_X(x,z)^2+2d_X(y,z)^2.
\end{equation}
If $(X,d_X)$ is a Hilbert space then, by the parallelogram identity, the inequalities~\eqref{eq:def alexandrov nonnegative} and~\eqref{eq:def hadamrd} hold true as equalities (with $w=(x+y)/2$). So, \eqref{eq:def alexandrov nonnegative} and~\eqref{eq:def hadamrd} are both natural relaxations of a stringent Hilbertian  identity (both relaxations have far-reaching implications). A complete Riemannian manifold is an Alexandrov space of nonnegative curvature if and only if its sectional curvature is nonnegative everywhere, and a complete simply connected Riemannian manifold is a Hadamard space if and only if its sectional curvature is nonpositive everywhere.

Following~\cite{Ott01}, it was shown in~\cite[Proposition~2.10]{Stu06} and~\cite[Appendix~A]{LV09} that $\mathscr{P}_{\!2}(\R^n)$ is an Alexandrov space of nonnegative curvature for every $n\in \N$; more generally, if $(X,d_X)$ is an Alexandrov space of nonnegative curvature then so is $\mathscr{P}_{\!2}(X)$. It therefore follows from Theorem~\ref{thm:main snowflake} that there exists an Alexandrov  space $(Y,d_Y)$ of nonnegative curvature that contains a bi-Lipschitz copy of the $1/2$-snowflake of every finite metric space, with distortion at most $1+\e$ for every $\e>0$. When this happens, we shall say that $(Y,d_Y)$ is $1/2$-snowflake universal.

\subsection{Subsets of Alexandrov spaces}\label{sec:char} It is a longstanding open problem, stated by Gromov in~\cite[Section $1.19_+$]{Gro99} and~\cite[\S15(b)]{Gro01}, as well as in, say, \cite{FLS07,AKP11,overflow},  to find an intrinsic characterization of those metric spaces that admit a bi-Lipschitz, or even isometric, embedding into an Alexandrov space of either nonnegative or nonpositive curvature.

Berg and Nikolaev~\cite{BN07,BN08} (see also~\cite{Sat09}) proved that a complete metric space $(X,d_X)$ is a Hadamard space if and only if it is geodesic and every $x_1,x_2,x_3,x_4\in X$ satisfy
\begin{equation}\label{eq:roundness 2}
d_X(x_1,x_3)^2+d_X(x_2,x_4)^2\le d_X(x_1,x_2)^2+d_X(x_2,x_3)^2+d_X(x_3,x_4)^2+d_X(x_4,x_1)^2.
\end{equation}
Inequality~\eqref{eq:roundness 2} is known in the literature under several names, including Enflo's ``roundness 2 property" (see~\cite{Enf69}), ``the short diagonal inequality" (see~\cite{Mat02}), or simply ``the quadrilateral inequality," and it has a variety of important  applications. Another characterization of this nature is due to Foertsch, Lytchak and Schroeder~\cite{FLS07}, who proved that a complete metric space $(X,d_X)$ is a Hadamard space if and only if it is geodesic, every  $x_1,x_2,x_3,x_4\in X$ satisfy the inequality
\begin{equation}\label{eq:potelmy intro}
d_X(x_1,x_3)\cdot d_X(x_2,x_4)\le d_X(x_1,x_2)\cdot d_X(x_3,x_4)+d_X(x_2,x_3)\cdot d_X(x_1,x_4),
\end{equation}
and if $w$ is a metric midpoint of $x_1$ and $x_2$ and $z$ is a metric midpoint of $x_3$ and $x_4$ then we have
\begin{equation}\label{eq:busemann}
d_X(w,z)\le \frac{d_X(x_1,x_3)+d_X(x_2,x_4)}{2}.
\end{equation}
\eqref{eq:potelmy intro} is called the Ptolemy inequality~\cite{FS11}, and condition~\eqref{eq:busemann} is called Busemann convexity~\cite{Bus55}.

Turning now to characterizations of nonnegative curvature, Lebedeva and Petrunin~\cite{LP10} proved that a complete metric space $(X,d_X)$ is an Alexandrov space of nonnegative curvature if and only if it is geodesic and every $x,y,z,w\in X$ satisfy
$$
d_X(x,w)^2+d_X(y,w)^2+d_X(z,w)^2\ge \frac{d_X(x,y)^2+d_X(x,z)^2+d_X(y,z)^2}{3}.
$$
Another (related) important characterization of Alexandrov spaces of nonnegative curvature asserts that a metric space $(X,d_X)$ is an Alexandrov spaces of nonnegative curvature if and only if it is geodesic and for every finitely supported $X$-valued random variable $Z$ we have
\begin{equation}\label{eq:LSS}
\E\big[d_X(Z,Z')^2\big]\le 2\inf_{x\in X}\E\big[d_X(Z,x)^2\big],
\end{equation}
where $Z'$ is an independent copy of $Z$. The above characterization  is due to Sturm~\cite{Stu99}, with the fact that nonnegative curvature in the sense of Alexandrov implies the validity of~\eqref{eq:LSS} being due to  Lang and Schroeder~\cite{LS97}. Following e.g.~\cite{Yok12}, condition~\eqref{eq:LSS} (which we shall use in Section~\ref{sec:sharpness}) is therefore called the Lang--Schroeder--Sturm inequality.

The above statements are interesting characterizations of spaces that are isometric to Alexandrov spaces of either nonpositive or nonnegative curvature, but they fail to characterize {\em subsets} of such spaces, since they require additional convexity properties of the metric space in question, such as being geodesic or Busemann convex. These assumptions are not intrinsic because they stipulate the existence of auxiliary points (metric midpoints) which may fall outside the given subset. Furthermore, these characterizations are {\em isometric} in nature, thus failing to address the important question of understanding when, given $D\in (1,\infty)$,  a metric space $(X,d_X)$ embeds with distortion at most $D$ into some Alexandrov space of either nonpositive or nonnegative curvature. One can search for such characterizations only among families of {\em quadratic metric inequalities}, as we shall now explain; in our context this is especially natural because the definitions~\eqref{eq:def alexandrov nonnegative} and~\eqref{eq:def hadamrd} are themselves quadratic.

\subsubsection{Quadratic metric inequalities}\label{sec:quadratic} For $n\in \N$ and $n$ by $n$  matrices $A=(a_{ij}), B=(b_{ij})\in M_n(\R)$  with nonnegative entries, say that a metric space $(X,d_X)$ satisfies the $(A,B)$-quadratic metric inequality if for every $x_1,\ldots,x_n\in X$ we have
\begin{equation*}\label{eq:quad met def AB}
\sum_{i=1}^n\sum_{j=1}^n a_{ij}d_X(x_i,x_j)^2\le \sum_{i=1}^n\sum_{j=1}^n b_{ij}d_X(x_i,x_j)^2.
\end{equation*}
The property of satisfying the $(A,B)$-quadratic metric inequality is clearly preserved by forming Pythagorean products, i.e., if $(X,d_X)$ and $(Y,d_Y)$ both satisfy the $(A,B)$-quadratic metric inequality then so does their Pythagorean product $(X\oplus Y)_2$. Here $(X\oplus Y)_2$ denotes the space $X\times Y$, equipped with the metric that is defined by
$$
\forall (a,b),(\alpha,\beta)\in X\times Y,\qquad d_{(X\oplus Y)_2}\big((a,b),(\alpha,\beta)\big)\eqdef \sqrt{d_X(a,\alpha)^2+d_Y(b,\beta)^2}.
$$
 The $(A,B)$-quadratic metric inequality is also preserved by ultraproducts (see e.g.~\cite[Section~2.4]{KL98} for background on ultraproducts of metric spaces), and it is a bi-Lipschitz invariant in the sense that if $(X,d_X)$ embeds with distortion at most $D\in [1,\infty)$ into $(Y,d_Y)$, and $(Y,d_Y)$ satisfies the $(A,B)$-quadratic metric inequality then $(X,d_X)$ satisfies the $(A,D^2B)$-quadratic metric inequality.

The following proposition is a converse to the above discussion.

\begin{proposition}\label{prop:duality}
Let $\mathscr{F}$ be a family of metric spaces that is closed under dilation and Pythagorean products, i.e., if $(U,d_U),(V,d_V)\in \mathscr{F}$ and $s\in (0,\infty)$ then also $(U,sd_U)\in \mathscr{F}$ and  $(U\oplus V)_2\in \mathscr{F}$. Fix $D\in [1,\infty)$ and $n\in \N$. Then an $n$-point metric space $(X,d_X)$ satisfies $$\inf_{(Y,d_Y)\in \mathscr{F}} c_Y(X)\le D$$ if and only if for every two $n$ by $n$ matrices $A,B\in M_n(\R)$ with nonnegative entries such that every $(Z,d_Z)\in \mathscr{F}$ satisfies the $(A,B)$-quadratic metric inequality, we also have that  $(X,d_X)$ satisfies the $(A,D^2B)$-quadratic metric inequality.
\end{proposition}

The proof of Proposition~\ref{prop:duality} appears in Section~\ref{sec:duality} below and consists of a duality argument that mimics the proof of Proposition~15.5.2 in~\cite{Mat02}, which deals with embeddings into Hilbert space.

\begin{remark}{\em It is a formal consequence of Proposition~\ref{prop:duality} that if the family of metric spaces $\mathscr{F}$  is also closed under ultraproducts, as are Alexandrov spaces with upper or lower curvature bounds (see e.g.~\cite[Section~2.4]{KL98}), then one does not need to restrict to finite metric spaces. Namely, in this case a metric space $(X,d_X)$  admits a bi-Lipschitz embedding into some $(Y,d_Y)\in \mathscr{F}$ if and only if there exists $D\in [1,\infty)$ such that $(X,d_X)$ satisfies the $(A,D^2B)$-quadratic metric inequality for every two $n$ by $n$ matrices $A,B\in M_n(\R)$ with nonnegative entries such that every $(Z,d_Z)\in \mathscr{F}$ satisfies the $(A,B)$-quadratic metric inequality.}
\end{remark}

\begin{remark}{\em
The Ptolemy inequality~\eqref{eq:potelmy intro} is not a quadratic metric inequality, yet it holds true in any Hadamard space. Proposition~\ref{prop:duality} implies that the Ptolemy inequality could be deduced from quadratic metric inequalities that hold true in Hadamard spaces. This is carried out explicitly in Section~\ref{sec:hadamard}  below, yielding an instructive proof (and  strengthening) of the Ptolemy inequality in Hadamard spaces that is conceptually different from its previously known proofs~\cite{FLS07,BFW09}.
}
\end{remark}

Theorem~\ref{thm:main snowflake} implies that all the quadratic metric inequalities that hold true in every Alexandrov space of nonnegative curvature ``trivialize" if one does not square the distances. Specifically, since $\mathscr{P}_{\!2}(\R^3)$ is an Alexandrov space of nonnegative curvature, the following statement is an immediate consequence of Theorem~\ref{thm:main snowflake}.

\begin{theorem}\label{thm:all about squares}
Suppose that $A,B\in M_n(\R)$ are $n$ by $n$ matrices with nonnegative entries such that every Alexandrov space of nonnegative curvature satisfies the $(A,B)$-quadratic metric inequality. Then for every metric space $(X,d_X)$ and every $x_1,\ldots,x_n\in X$ we have
\begin{equation}\label{eq:quadratic trivialized}
\sum_{i=1}^n\sum_{j=1}^n a_{ij} d_X(x_i,x_j)\le \sum_{i=1}^n\sum_{j=1}^n b_{ij}d_X(x_i,x_j).
\end{equation}
\end{theorem}
While Theorem~\ref{thm:all about squares} does not answer the question of characterizing those quadratic metric inequalities that hold true in any Alexandrov space of nonnegative curvature, it does show that such inequalities rely crucially on the fact that distances are being squared, i.e., if one removes the squares then one arrives at an inequality~\eqref{eq:quadratic trivialized} which must be nothing more than a consequence of the triangle inequality.

Obtaining a full characterization of those quadratic metric inequalities that hold true in any Alexandrov space of nonnegative curvature remains an important challenge. Many such inequalities are known, including, as shown by Ohta~\cite{Oht09}, Markov type 2 (note, however, that the supremum of the  Markov type $2$ constants of all Alexandrov spaces of nonnegative curvature is an unknown universal constant~\cite{OP09}; we obtain the best known bound on this constant in   Corollary~\ref{eq:Mtype 2 const} below). Another family of nontrivial quadratic metric inequalities that hold true in any Alexandrov space of nonnegative curvature is obtained in~\cite{AN09}, where it is shown  that all such spaces have Markov convexity $2$. By these observations combined with the nonlinear Maurey--Pisier theorem~\cite{MN08}, we know that there exists $q<\infty$ such that any Alexandrov space of nonnegative curvature has metric cotype $q$. It is natural to conjecture that one could take $q=2$ here, but at present this remains open. For more on the notions discussed above, i.e., Markov type, Markov convexity and metric cotype, as well as their applications, see the survey~\cite{Nao12} and the references therein.

The above discussion in the context of Hadamard spaces remains an important open problem. At present we do not know of any metric space $(X,d_X)$ such that the metric space $(X,\sqrt{d_X})$ fails to admit a bi-Lipschitz embedding into some Hadamard space. More generally, while a variety of nontrivial quadratic metric inequalities are known to hold true in any Hadamard space, a full characterization of such inequalities remains elusive. In Section~\ref{sec:hadamard} below we formulate a systematic way to generate such inequalities, posing the question whether the hierarchy of inequalities thus obtained yields a characterization of those metric spaces that admit a bi-Lipschitz embedding into some Hadamard space.

\subsubsection{Uniform, coarse and quasisymmetric embeddings}\label{sec:uniform}
A metric space $(X,d_X)$ is said to embed uniformly into a metric space $(Y,d_Y)$ if there exists an injection $f:X\to Y$ such that both $f$ and $f^{-1}$ are uniformly continuous. $(X,d_X)$ is said~\cite{Gro93} to embed coarsely into $(Y,d_Y)$ if there exists $f:X\to Y$ and nondecreasing functions $\alpha,\beta:[0,\infty)\to [0,\infty)$ with $\lim_{t\to \infty}\alpha(t)=\infty$ such that
\begin{equation}\label{eq:def coarse}
\forall\, x,y\in  X,\qquad \alpha(d_X(x,y))\le d_Y(f(x),f(y))\le \beta(d_X(x,y)).
\end{equation}
$(X,d_X)$ is said~\cite{BA56,TV80} to admit a quasisymmetric embedding into $(Y,d_Y)$ if there exists an injection $f:X\to Y$ and $\eta:(0,\infty)\to (0,\infty)$ with $\lim_{t\to 0}\eta(t)=0$ such that for every distinct $x,y,z\in X$,
$$
\frac{d_Y(f(x),f(y))}{d_Y(f(x),f(z))}\le \eta\left(\frac{d_X(x,y)}{d_X(x,z)}\right).
$$

A direct combination of Theorem~\ref{thm:main snowflake} with the results of~\cite{MN08,Nao12-quasi} shows that $\mathscr{P}_{\!2}(\R^3)$ does not embed even in the above weak senses into any Banach space of nontrivial (Rademacher) type; we refer to the survey~\cite{Mau03} and the references therein for more on the notion of type of Banach spaces. In particular, $\mathscr{P}_{\!2}(\R^3)$ fails to admit such embeddings into any $L_p(\mu)$ space for finite $p$ (for the case $p=1$, use the fact that the $1/2$-snowflake of an $L_1(\mu)$ space embeds isometrically into a Hilbert space; see~\cite{WW75}), or, say, into any uniformly convex Banach space. It remains an interesting open question whether or not these assertions also hold true for $\mathscr{P}_{\!2}(\R^2)$.

\begin{theorem}\label{thm:no embed quasi} If $p>1$ then
$\mathscr{P}_{\!p}(\R^3)$ does not admit a uniform, coarse or quasisymmetric embedding into any Banach space of nontrivial type.
\end{theorem}
Note that a positive resolution of a key conjecture of~\cite{MN08}, namely the first question in Section~8 of~\cite{MN08},  would ``upgrade" Theorem~\ref{thm:no embed quasi} to the (best possible) assertion that $\mathscr{P}_{\!2}(\R^3)$ does not admit a uniform, coarse or quasisymmetric embedding into any Banach space of finite cotype.

\begin{remark}\label{rem:torus sum}{\em
Very few other examples of Alexandrov spaces of nonnegative curvature with poor embeddability properties into Banach spaces are known, all of which are not known to satisfy properties as strong as the conclusion of  Theorem~\ref{thm:no embed quasi}. Specifically, in~\cite{AN09} it is shown that $\mathscr{P}_{\!2}(\R^2)$ fails to admit a bi-Lipschitz embedding into $L_1$. A construction with stronger properties follows from the earlier work~\cite{KN06}, combined with the recent methods of~\cite{Nao14}. Specifically, it follows from~\cite{KN06} and~\cite{Nao14} that for every $n\in \N$ there exists a lattice $\Lambda_n\subset \R^n$ of rank $n$ such that if we consider  the following infinite Pythagorean product of flat tori
 \begin{equation}\label{eq:def torus sum}
 \mathscr{T}\eqdef  \bigg(\bigoplus_{n=1}^\infty \R^n/\Lambda_n\bigg)_{\!2},
 \end{equation}
 then $\mathscr{T}$ fails to admit a uniform or coarse embedding into a certain class of Banach spaces that includes all Banach lattices of finite cotype and all the noncommutative $L_p$ spaces for finite $p\ge 1$.  Since for every $n\in \N$ the sectional curvature of $\R^n/\Lambda_n$ vanishes, it is an Alexandrov space of nonnegative curvature, and therefore so is the Pythagorean product $\mathscr{T}$. It remains an interesting open question whether or not $\mathscr{T}$ admits a uniform, coarse or quasisymmetric embedding into some Banach space of nontrivial type, and, for that matter, even whether or not $\mathscr{T}$ is $1/2$-snowflake universal. We speculate that the answer to the latter question is negative.}
 \end{remark}

\subsubsection{Expanders with respect to Alexandrov spaces}\label{sec:expanders} Fixing an integer $k\ge 3$, an unbounded sequence of $k$-regular finite  graphs $\{(V_{j},E_{j})\}_{j=1}^\infty$ is said to be an expander with respect to a metric space $(X,d_X)$ if  for every $j\in \N$ and $\{x_u\}_{u\in V_j}\subset X$ we have
\begin{equation}\label{eq:expander def}
\frac{1}{|V_{j}|^2}\sum_{(u,v)\in V_{j}\times V_j} d_X(x_u,x_v)^2\asymp_X \frac{1}{|E_j|}\sum_{\{u,v\}\in E_{j}} d_X(x_u,x_v)^2.
\end{equation}
Unless $X$ is a singleton, a sequence of expanders with respect to $(X,d_X)$ must also be a sequence of expanders in the classical (combinatorial) sense. See~\cite{NS11,MN14,MN-duke,Nao14,Now15} and the references therein for background on expanders with respect to metric spaces and their applications.

In contrast to the case of classical expanders, the question of understanding when a metric space $X$ admits an expander sequence seems to be very difficult (even in the special case when $X$ is a Banach space), with limited availability of methods~\cite{Mat97,Oza04,Laf08,Laf09,MN14,Lia14,Nao14,MN-duke,Mim15} for establishing metric inequalities such as~\eqref{eq:expander def}. Theorem~\ref{thm:main snowflake} implies that $\mathscr{P}_{\!p}(\R^3)$ fails to admit a sequence of expanders for every $p\in (1,\infty)$. The particular case $p=2$ establishes for the first time the (arguably surprising) fact that there exists an Alexandrov space of nonnegative curvature with respect to which expanders do not exist.

\begin{theorem}\label{thm:no expander} For $p>1$ no sequence of bounded degree graphs is an expander with respect to $\Pp(\R^3)$.
\end{theorem}
To deduce Theorem~\ref{thm:no expander} from Theorem~\ref{thm:main snowflake}, use an argument of Gromov~\cite{Gromov-random} (which is reproduced in~\cite[Section~1.1]{MN14}), to deduce that if  $\{G_n=(V_{n},E_{n})\}_{n=1}^\infty$  were a $k$-regular expander with respect to $\Pp(\R^3)$ then, denoting the shortest-path metric that $G_n$ induces on $V_n$ by $d_n$ (the assumption that $G_n$ is an expander with respect to a non-singleton metric space implies that it is a classical expander, hence connected), the metric spaces $\{(V_n,d_n)\}_{n=1}^\infty$ fail to admit a coarse embedding into $\Pp(\R^3)$ with any moduli $\alpha,\beta:[0,\infty)\to [0,\infty)$ as in~\eqref{eq:def coarse} that are independent of $n$. This contradicts the fact that by Theorem~\ref{thm:main snowflake} we know that for every $n\in \N$ the finite  metric space $(V_n,d_n)$ embeds coarsely into $\Pp(\R^3)$ with moduli $\alpha(t)=t^{1/p}$ and, say, $\beta(t)=2t^{1/p}$.

 The above question for Hadamard spaces remains an important open problem which goes back at least to~\cite{Gromov-random,NS11}. See~\cite{MN-duke} for more on this theme, where it is shown that there exists a Hadamard space with respect to which random regular graphs are asymptotically almost surely not expanders.  We also ask whether or not the Alexandrov space of nonnegative curvature $\mathscr{T}$ of Remark~\ref{rem:torus sum} admits a sequence of bounded degree expanders; we speculate that it does.

\subsection{The universality problem for $\mathscr{P}_{\!1}(\R^k)$}\label{sec:bourgain}  A metric space $(Y,d_Y)$ is said to be (finitely) universal if there exists $K\in (0,\infty)$ such that $c_Y(X)\le K$ for {\em every} {\em finite} metric space $(X,d_X)$.

In~\cite{Bou86} Bourgain asked whether $(\mathscr{P}_{\! 1}(\R^2),\W_1)$ is not universal. He actually formulated this question as asking whether a certain Banach space (namely, the dual of the Lipschitz functions on the square $[0,1]^2$), which we denote for the sake of the present discussion by $Z$, has finite Rademacher cotype, but this is equivalent to the above formulation in terms of the  universality of $(\mathscr{P}_{\! 1}(\R^2),\W_1)$. It is not necessary to be familiar with the notion of cotype in order to understand the ensuing discussion, so readers can consider only the above formulation of Bourgain's question. However, for experts we shall now briefly justify this equivalence.  For Banach spaces the property of not being universal is equivalent to having finite Rademacher cotype, as follows from Ribe's theorem~\cite{Rib76} and the Maurey--Pisier theorem~\cite{MP76}. As explained in~\cite{NS07}, every finite subset of $Z$ embeds into $\mathscr{P}_{\! 1}(\R^2)$ with distortion arbitrarily close to $1$, and, conversely, every finite subset of $\mathscr{P}_{\! 1}(\R^2)$ embeds into $Z$ with  distortion arbitrarily close to $1$. Hence $Z$ is universal if and only if $\mathscr{P}_{\! 1}(\R^2)$ is universal. So, $Z$ has finite Rademacher cotype if and only if $\mathscr{P}_{\! 1}(\R^2)$ is not universal.

Bourgain proved in~\cite{Bou86} that $(\mathscr{P}_{\! 1}(\ell_1),\W_1)$ is universal (despite the fact that $\ell_1$ is not universal), but it remains an intriguing open question to determine whether or not $(\mathscr{P}_{\! 1}(\R^k),\W_1)$ is universal for any finite $k\in \N$, the case $k=2$ being most challenging. Here we show that Wasserstein spaces do exhibit some universality phenomenon even when the underlying metric space is a finite dimensional Euclidean space, but we fall short of addressing the universality problem for $\mathscr{P}_{\!1}(\R^k)$. Specifically, Theorem~\ref{thm:main snowflake} asserts that $(\Pp(\R^3),\W_p)$ is universal with respect to $1/p$-snowflakes of metric spaces, and if $p\in (1,2]$ then this cannot be improved to $\alpha$-snowflakes for any $\alpha>1/p$, by Theorem~\eqref{thm:no large snowflake}.   The $1/p$-snowflake of $(X,d_X)$ becomes ``closer" to $(X,d_X)$ itself as $p\to 1$, and at the same time $(\Pp(\R^3),\W_p)$ becomes ``closer" to    $(\mathscr{P}_{\! 1}(\R^3),\W_1)$, but Theorem~\ref{thm:main snowflake} fails to imply the universality of $(\mathscr{P}_{\! 1}(\R^3),\W_1)$ because the embeddings that we construct in Theorem~\ref{thm:main snowflake} degenerate as $p\to 1$.

\begin{remark}
{\em The universality problem for $\mathscr{P}_{\!1}(\R^k)$ belongs to longstanding traditions in functional analysis. As Bourgain explains in~\cite{Bou86}, one motivation for his question is an idea of W.~B.~Johnson to ``linearize" bi-Lipschitz classification problems by examining the geometry of the corresponding Banach spaces of Lipschitz functions defined on the metric spaces in question. For this ``functorial linearization" to succeed, one needs to sufficiently understand the linear structure of the spaces of Lipschitz functions on metric spaces, but unfortunately these are wild spaces that are poorly understood. The  universality problem for $\mathscr{P}_{\!1}(\R^k)$ highlights this situation by asking a basic geometric question (universality) about the dual of the space of Lipschitz functions on $\R^k$.  Despite these difficulties, in recent years the above approach to bi-Lipschitz classification problems has been successfully developed, notably by Godefroy and Kalton~\cite{GK03} who, among other results, deduced from this approach that the Bounded Approximation Property (BAP) is preserved under bi-Lipschitz homeomorphisms of Banach spaces. In addition to being motivated by potential applications, the universality problem for $\mathscr{P}_{\!1}(\R^k)$ relates to old questions on the structure of classical function spaces: here the spaces in question are the Lipschitz functions on $\R^k$, which are closely related to the spaces $C^1(\R^k)$ whose linear structure (in particular its dependence on $k$) remains a major mystery that goes back to Banach's seminal work. Understanding the universality of classical Banach spaces and their duals has attracted many efforts over the past decades, notable examples of which include work~\cite{Pel77,Bou84} on the (non)universality of the dual of the Hardy space $H^\infty(S^1)$, work~\cite{Var76,Pis78,KP80,BM85} on the universality of the span in $C(G)$ of a subset of characters of a compact Abelian group $G$, and work~\cite{Tom74,Pis92,BNR12} on the universality of projective tensor products. Despite these efforts, understanding the universality of $\mathscr{P}_{\!1}(\R^k)$ (equivalently, whether or not the dual of the space of Lipschitz functions on $\R^k$ has finite cotype) remains a remarkably stubborn open problem.
}
\end{remark}

Our proof of Theorem~\ref{thm:main snowflake} relies on the fact that the underlying Euclidean space is (at least) $3$-dimensional, so it remains open whether or not, say, the $1/2$-snowflake of every finite metric space embeds with $O(1)$  distortion into $(\mathscr{P}_{\! 2}(\R^2),\W_2)$. In~\cite{AN09} it is proved that every finite subset of the metric space $(\mathscr{P}_{\!1}(\R^2),\sqrt{\W_1})$, i.e., the $1/2$-snowflake of $(\mathscr{P}_{\! 1}(\R^2),\W_1)$, embeds with $O(1)$ distortion into $(\mathscr{P}_{\! 2}(\R^2),\W_2)$. Thus, if $(\mathscr{P}_{\! 1}(\R^2),\W_1)$ were universal (i.e., if the universality problem for $\mathscr{P}_{\!1}(\R^2)$ had a negative answer) then it would follow  that the $1/2$-snowflake of every finite metric space embeds with $O(1)$ distortion into $(\mathscr{P}_{\! 2}(\R^2),\W_2)$.

\begin{remark}\label{rem:analogy with lp} {\em Another interesting open question is whether or not $\mathscr{P}_{\!1}(\R^3)$ (or $\mathscr{P}_{\!1}(\R^2)$ for that matter) is $1/2$-snowflake universal. There is a perceived analogy between the spaces $\Pp(X)$ and $L_p(\mu)$ spaces, with the spaces $\Pp(X)$ sometimes being referred to as the geometric measure theory analogues of $L_p(\mu)$ spaces. It would be very  interesting to investigate whether or not this analogy could be put on firm footing. As an example of a concrete question along these lines, since $L_2$ is isometric to a subspace of $L_p$, we ask for a characterization of those metric spaces $X$ for which $\mathscr{P}_{\!2}(X)$ admits a bi-Lipschitz embedding into $\Pp(X)$, or, less ambitiously, when does there exist $D(X)\in [1,\infty)$ such that every finite subset of $\mathscr{P}_{\!2}(X)$ embeds into $\Pp(X)$ with distortion $D(X)$. If this were true when $X=\R^3$ or $X=\R^2$ (it is easily seen to be true when $X=\R$) and $p=1$ then it would follow from Theorem~\ref{thm:main snowflake} that $\mathscr{P}_{\!1}(\R^3)$ (respectively $\mathscr{P}_{\!1}(\R^2)$) is $1/2$-snowflake universal. By~\cite{MN08}, this, in turn, would imply that $\mathscr{P}_{\!1}(\R^3)$ (respectively $\mathscr{P}_{\!1}(\R^2)$) fails to admit a coarse, uniform or quasisymmetric embedding into $L_1$, thus strengthening results of~\cite{NS07} via an approach that is entirely different from that of~\cite{NS07}. There are many additional  open questions that follow from the analogy between Wasserstein $p$ spaces and $L_p(\mu)$ spaces, including various questions about the evaluation of the metric type and cotype of $\Pp(X)$; see Question~\ref{Q:p>2} below for more on this interesting research direction.
}
\end{remark}

\subsubsection{Towards the metric cotype dichotomy problem} The following theorem was proved in~\cite{MN08}; see~\cite{Men09,MN11-arxiv,MN13-JEMS} for more information on metric dichotomies of this type.

\begin{theorem}[Metric cotype dichotomy~\cite{MN08}]\label{thm:cotype dichotomy}
Let $(X,d_X)$ be a metric space that isn't universal. There exists $\alpha(X)\in (0,\infty)$ and finite metric spaces $\{(M_n,d_{M_n})\}_{n=1}^\infty$ with $\lim_{n\to \infty} |M_n|=\infty$ and
$$
\forall\, n\in \N,\qquad c_X(M_n)\ge (\log |M_n|)^{\alpha(X)}.
$$
\end{theorem}
A central question that was left open in~\cite{MN08}, called the {\em metric cotype dichotomy problem}, is whether the exponent $\alpha(X)\in (0,\infty)$ of Theorem~\ref{thm:cotype dichotomy} can be taken to be a universal constant, i.e.,

\begin{question}[Metric cotype dichotomy problem~\cite{MN08}]\label{Q:cotype dichotomy}
Does there exist $\alpha\in (0,1]$ such that {\em every} non-universal metric space $X$ admits a sequence of finite metric spaces $\{(M_n,d_{M_n})\}_{n=1}^\infty$ with $\lim_{n\to \infty} |M_n|=\infty$ that satisfies $c_X(M_n)\ge (\log |M_n|)^\alpha$?
\end{question}
It is even unknown whether or not in Question~\ref{Q:cotype dichotomy} one could take $\alpha=1$ (by Bourgain's embedding theorem~\cite{Bou85}, the best one could hope for here is $\alpha=1$). A positive answer to the following question would resolve the metric cotype dichotomy problem negatively; this question corresponds to asking if Theorem~\ref{thm:no large snowflake}  is sharp when $p\in (1,2]$ and $\alpha=1$ (the same question when $\alpha \in (1/p,1)$ is also open).

\begin{question}\label{Q:dichotomy sharp wasserstein}
Is it true that for $p\in (1,2]$ and $n\in \N$ every $n$-point metric space $(X,d_X)$ satisfies
$$
c_{\Pp(\R^3)}(X)\lesssim_p (\log n)^{1-\frac{1}{p}}?
$$
\end{question}
 A positive answer to Question~\eqref{Q:dichotomy sharp wasserstein} would imply that  $\alpha(\Pp(\R^3))\le 1-1/p$, using the notation of Theorem~\ref{thm:cotype dichotomy}. Taking $p\to 1^+$, it would therefore follow that there is no $\alpha>0$ as in Question~\ref{Q:cotype dichotomy}.

We believe that Question~\ref{Q:dichotomy sharp wasserstein} is an especially intriguing   challenge in embedding theory (for a concrete and natural target space) because a positive answer, in addition to resolving the metric cotype dichotomy problem, would require an interesting new construction, and  a negative answer would require devising a new bi-Lipschitz invariant that would serve as an obstruction for embeddings into Wasserstein spaces.

Focusing for concreteness on the case $p=2$, Question~\ref{Q:dichotomy sharp wasserstein} asks whether $c_{\mathscr{P}_{\!2}(\R^3)}(X)\lesssim \sqrt{\log n}$ for every $n$-point metric space $(X,d_X)$. Note that Theorem~\ref{thm:main snowflake} implies that $(X,d_X)$ embeds into $\mathscr{P}_{\!2}(X)$ with distortion at most the square root of the {\em aspect ratio} of $(X,d_X)$, i.e.,
\begin{equation}\label{eq:aspect}
c_{(\mathscr{P}_{\!2}(\R^3),\W_2)}(X,d_X)\le \sqrt{\frac{\diam(X,d_X)}{\min_{\substack{x,y\in X\\ x\neq y}}d_X(x,y)}},
\end{equation}
but we are asking here for the largest possible growth rate of the distortion of $X$ into $\mathscr{P}_{\!2}(X)$ in terms of the cardinality of $X$. While for certain embedding results there are standard methods (see e.g.~\cite{Bar96,HM06,MN10-max}) for replacing  the dependence  on the aspect ratio of a finite metric space by a dependence on its cardinality, these methods do not seem to apply to our embedding in~\eqref{eq:aspect}. See Section~\ref{sec:dichotomy question remarks} below for further discussion.

\section{Proof of Theorem~\ref{thm:main snowflake}}\label{sec:main proof}

In what follows fix $n\in \N$ and an $n$-point metric space $(X,d_X)$. Write $X=\{x_1,\ldots,x_n\}$ and fix $\phi:\n\times \n\to \{1,\ldots,n^2\}$ to be an arbitrary bijection between $\n\times \n$ and $\{1,\ldots,n^2\}$. Below it will be convenient to use the following notation.
\begin{equation}\label{eq:mM}
m\eqdef \min_{\substack{x,y\in X\\ x\neq y}} d_X(x,y)^{\frac{1}{p}}\qquad\mathrm{and}\qquad M\eqdef\max_{x,y\in X} d_X(x,y)^{\frac{1}{p}}.
\end{equation}

Fix $K\in \N$. Denoting the standard basis of $\R^3$ by $e_1=(1,0,0)$, $e_2=(0,1,0)$, $e_3=(0,0,1)$, for every $i,j\in \n$ with $i<j$ define  five families of points in $\R^3$ by setting for  $s\in \{0,\ldots,K\}$,
\begin{align}
Q^1_s(i,j)&\eqdef \frac{Mi}{m}e_1+\frac{M\phi(i,j)s}{mK}e_2,\label{eq:Q1}\\
Q^2_s(i,j)&\eqdef \frac{Mi}{m}e_1+\frac{M\phi(i,j)}{m}e_2+\frac{Ms}{mK}e_3,\label{eq:Q2}\\
\begin{split}
Q^3_s(i,j)&\eqdef \frac{M(s(j-i)+Ki)+(K-s)d_X(x_i,x_j)^{\frac{1}{p}}}{mK}e_1   +\frac{M\phi(i,j)}{m}e_2+\frac{M}{m}e_3,\label{eq:Q3}
\end{split}
\\
Q^4_s(i,j)&\eqdef \frac{Mj}{m}e_1+\frac{M\phi(i,j)}{m}e_2+\frac{M(K-s)}{mK}e_3,\label{eq:Q4}\\
Q^5_s(i,j)&\eqdef \frac{Mj}{m}e_1+\frac{M(K-s)\phi(i,j)}{mK}e_2\label{eq:Q5}.
\end{align}
Then $Q^1_K(i,j)=Q^2_0(i,j)$, $Q^3_K(i,j)=Q^4_0(i,j)$ and $Q^4_K(i,j)=Q^5_0(i,j)$, so the total number of points thus obtained equals $5(K+1)-3=5K+2$.

Define $\mathscr{B}\subset \R^3$ by setting
\begin{equation}\label{eq:def B}
\mathscr{B}\eqdef \bigcup_{\substack{i,j\in \n\\ i<j}}\mathscr{B}_{ij},
\end{equation}
where for every $i,j\in \n$ with $i<j$ we write
\begin{equation}\label{eq:def Bij}
\mathscr{B}_{ij}\eqdef \bigcup_{s=0}^K\left\{Q^1_s(i,j),Q^2_s(i,j),Q^3_s(i,j),Q^4_s(i,j),Q^5_s(i,j)\right\}.
\end{equation}
Hence $|\mathscr{B}_{ij}|=5K+2$. We also define $\mathscr{C}\subset \R^3$ by
\begin{equation}\label{eq:defC}
\mathscr{C}\eqdef \mathscr{B}\setminus \left\{\frac{Mi}{m}e_1:\  i\in \n\right\}.
\end{equation}
Note that by~\eqref{eq:Q1} we have $(Mi/m)e_1=Q^1_0(i,j)$ if $i,j\in \n$ satisfy $i<j$, and by~\eqref{eq:Q5} we have $(Mi/m)e_1=Q^5_K(\ell,i)$ if $\ell,i\in \n$ satisfy $\ell<i$. Thus $\mathscr{C}$ corresponds to removing from $\mathscr{B}$ those points that lie on the $x$-axis. In what follows, we denote $N=|\mathscr{C}|+1$.
Finally, for every $i\in \n$ we define $\mathscr{C}_i\subset \R^3$ by
\begin{equation}\label{eq:def Ci}
\mathscr{C}_i\eqdef \mathscr{C}\cup  \left\{\frac{Mi}{m}e_1\right\}.
\end{equation}
Hence $|\mathscr{C}_i|=N$.
Our embedding $f:X\to \Pp(\R^3)$ will be given by
\begin{equation}\label{eq:def f}
\forall\, j\in \n,\qquad f(x_j)\eqdef \frac{1}{N}\sum_{u\in \mathscr{C}_j} \delta_u,
\end{equation}
where, as usual, $\delta_u$ is the point mass at $u$.  Thus $f(x_j)$ is the uniform probability measure over $\mathscr{C}_j$. A schematic depiction of the above construction appears in Figure~\ref{fig:4points} below.

\medskip
\begin{figure}[here]
\centering
\fbox{
\begin{minipage}{6.25in}
\centering
\includegraphics[width=6.25in]{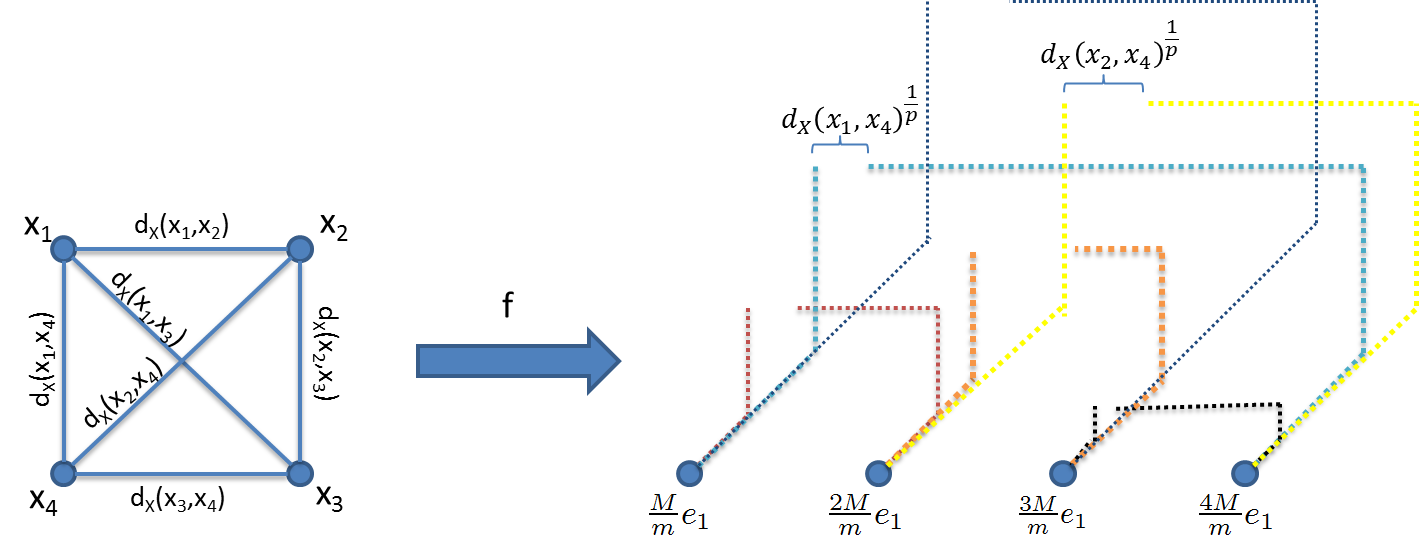}
\smallskip
\caption{A schematic depiction of the embedding $f:X\to \Pp(\R^3)$ for a four-point metric space $(X,d_X)=(\{x_1,x_2,x_3,x_4\},d_X)$. Here the $x$-axis is the horizontal direction, the $z$-axis is the vertical direction and the $y$-axis is perpendicular to the page plane. Recall that $m$ and $M$ are defined in~\eqref{eq:mM}.
}
\label{fig:4points}
\end{minipage}
}
\end{figure}

Lemma~\ref{lem:dist f} below estimates the distortion of $f$, proving Theorem~\ref{thm:main snowflake}.
\begin{lemma}\label{lem:dist f}
Fix $\e\in (0,1)$ and $p\in (1,\infty)$. Let $f:X\to \Pp(\R^3)$ be the mapping appearing in~\eqref{eq:def f}, considered as a mapping from the snowflaked metric space $(X,d_X^{1/p})$ to the metric space $(\Pp(\R^3),\W_p)$. Then, recalling the definitions of $m$ and $M$ in~\eqref{eq:mM}, we have
\begin{equation}\label{eq:K assumption}
K\ge \left(\frac{5M^pn^{2p}}{pm^p\e}\right)^{\frac{1}{p-1}}\implies \dist(f)\le 1+\e.
\end{equation}
\end{lemma}

\begin{proof}
We shall show that under the assumption on $K$ that appears in~\eqref{eq:K assumption} we have
\begin{equation}\label{eq:1+plus eps}
\forall\, i,j\in \n,\qquad \left(\frac{d_X(x_i,x_j)}{m^pN}\right)^{\frac{1}{p}} \le \W_p(f(x_i),f(x_j))\le (1+\e)\left(\frac{d_X(x_i,x_j)}{m^pN}\right)^{\frac{1}{p}},
\end{equation}
where we recall that we defined  $N$ to be equal to $|\mathscr{C}|+1$ for $\mathscr{C}$ given in~\eqref{eq:defC}. Clearly~\eqref{eq:1+plus eps} implies that $\dist(f)\le 1+\e$, as required.

To prove the right hand inequality in~\eqref{eq:1+plus eps}, suppose that $i,j\in \n$ satisfy $i<j$ and consider the coupling $\pi\in \Pi(f(x_i),f(x_j))$ given by
\begin{equation}\label{eq:first coupling formula}
\pi\eqdef \frac{1}{N}\bigg(\sum_{t=1}^5\sum_{s=0}^{K-1} \delta_{\left(Q^t_s(i,j),Q^t_{s+1}(i,j)\right)}+\delta_{\left(Q^2_K(i,j),Q^3_{0}(i,j)\right)}+\sum_{u\in\mathscr{C}\setminus \mathscr{B}_{ij}}\delta_{(u,u)}\bigg),
\end{equation}
where for~\eqref{eq:first coupling formula} recall~\eqref{eq:def Bij} and~\eqref{eq:defC}.  The meaning of~\eqref{eq:first coupling formula} is simple: the supports of $f(x_i)$ and $f(x_j)$ equal $\mathscr{C}_i$ and $\mathscr{C}_j$, respectively, where we recall~\eqref{eq:def Ci}. Note that $\mathscr{C}_i\setminus \mathscr{C}_j=\{Q^1_0(i,j)\}$  and  $\mathscr{C}_j\setminus \mathscr{C}_i=\{Q^5_K(i,j)\}$, where we recall~\eqref{eq:Q1} and~\eqref{eq:Q5}. So, the coupling $\pi$ in~\eqref{eq:first coupling formula} corresponds to shifting the points in $\mathscr{B}_{ij}$ from the support of $f(x_i)$ to the support of $f(x_j)$ while keeping the points in $\mathscr{C}\setminus \mathscr{B}_{ij}$ unchanged.

Now, recalling the definitions~\eqref{eq:Q1}, \eqref{eq:Q2}, \eqref{eq:Q3}, \eqref{eq:Q4} and~\eqref{eq:Q5},
\begin{multline}\label{eq:break into sum}
\W_p(f(x_i),f(x_j))^p\le \iint_{\R^3\times \R^3} \|x-y\|_2^p\d\pi(x,y) \\= \frac{1}{N}\sum_{t=1}^5\sum_{s=0}^{K-1} \left\|Q^t_s(i,j)-Q^t_{s+1}(i,j)\right\|_2^p+\frac{\|Q^2_K(i,j)-Q^3_{0}(i,j)\|_2^p}{N}.
\end{multline}
Note that if $s\in \{0,\ldots,K-1\}$ then by~\eqref{eq:Q1}, \eqref{eq:Q2}, \eqref{eq:Q4}, \eqref{eq:Q5} we have
\begin{align}\label{eq:first type summand}
\begin{split}
t\in \{1,5\}&\implies \left\|Q^t_s(i,j)-Q^t_{s+1}(i,j)\right\|_2=\frac{M\phi(i,j)}{mK}\le \frac{Mn^2}{mK},\\
t\in \{2,4\}&\implies \left\|Q^t_s(i,j)-Q^t_{s+1}(i,j)\right\|_2=\frac{M}{mK}.
\end{split}
\end{align}
Also, by~\eqref{eq:Q2} and~\eqref{eq:Q3} we have
\begin{equation}\label{eq:second type summand}
\left\|Q^2_K(i,j)-Q^3_{0}(i,j)\right\|_2=\frac{d_X(x_i,x_j)^{\frac{1}{p}}}{m}.
\end{equation}
Finally, by~\eqref{eq:Q3} for every $s\in \{0,\ldots,K-1\}$ we have
\begin{equation}\label{eq:third type summand}
\left\|Q^3_s(i,j)-Q^3_{s+1}(i,j)\right\|_2=\frac{M(j-i)}{mK}-\frac{d_X(x_i,x_j)^{\frac{1}{p}}}{mK}\le \frac{Mn}{mK},
\end{equation}
where we used the fact that $M(j-i)-d_X(x_i,x_j)^{1/p}\ge 0$, which holds true by the definition of $M$ in~\eqref{eq:mM} because $j-i\ge 1$. A substitution of~\eqref{eq:first type summand}, \eqref{eq:second type summand} and~\eqref{eq:third type summand} into~\eqref{eq:break into sum} yields the estimate
\begin{multline*}
\W_p(f(x_i),f(x_j))^p\le \frac{d_X(x_i,x_j)}{m^pN}+\frac{5K}{N}\left(\frac{Mn^2}{mK}\right)^p\\=\left(1+\frac{5M^pn^{2p}}{K^{p-1}d_X(x_i,x_j)}\right)\frac{d_X(x_i,x_j)}{m^pN}
\le (1+p\e)\frac{d_X(x_i,x_j)}{m^pN},
\end{multline*}
where we used the fact that by the definition of $m$ in~\eqref{eq:mM} we have $m^p\le d_X(x_i,x_j)$, and the lower bound on $K$ that is assumed in~\eqref{eq:K assumption}.  This implies the right hand inequality in~\eqref{eq:1+plus eps} because $1+p\e\le (1+\e)^p$.

Passing now to the proof of the left hand inequality in~\eqref{eq:1+plus eps}, we need to prove that for every $i,j\in \n$ with $i<j$ we have
\begin{equation}\label{eq:lower pi}
\forall\, \pi\in \Pi(f(x_i),f(x_j)),\qquad \iint_{\R^3\times \R^3} \|x-y\|_2^p\d\pi(x,y)\ge \frac{d_X(x_i,x_j)}{m^pN}.
\end{equation}
Note that we still did not use the triangle inequality for $d_X$, but this will be used in the proof of~\eqref{eq:lower pi}. Also, the reason why we are dealing with $\Pp(\R^3)$ rather than $\Pp(\R^2)$ will become clear in the ensuing argument.

Recall that the measures $f(x_i)$ and $f(x_j)$ are uniformly distributed over sets of the same size, and their supports $\mathscr{C}_i$ and $\mathscr{C}_j$ (respectively) satisfy $\mathscr{C}_i\triangle \mathscr{C}_j=\{(Mi/m)e_1,(Mj/m)e_1\}$. Since the set of all doubly stochastic matrices is the convex hull of the permutation matrices, and every permutation is a product of disjoint cycles, it follows that it suffices to establish the validity of~\eqref{eq:lower pi} when $\pi=\frac{1}{N}\sum_{\ell=1}^{L}\delta_{\left(u_{\ell-1},u_\ell\right)}$ for some $L\in \n$ and $u_1,\ldots u_{L-1}\in \mathscr{C}$, where we set $u_0=(Mi/m)e_1$ and $u_{L}=(Mj/m)e_1$. With this notation, our goal is to show that
\begin{equation}\label{eq:goal with sum}
\frac{1}{N} \sum_{\ell=1}^L \|u_\ell-u_{\ell-1}\|_2^p\ge \frac{d_X(x_i,x_j)}{m^pN}.
\end{equation}

For every $a\in \n$ define $\mathscr{S}_a\subset \R^3$ by $\mathscr{S}_a\eqdef \mathscr{S}_a^1\cup \mathscr{S}_a^2$, where
\begin{equation}\label{eq:S1}
\mathscr{S}_a^1\eqdef \bigcup_{b=a+1}^n\bigcup_{s=0}^K\left\{Q_s^1(a,b),Q_s^2(a,b)\right\},
\end{equation}
and
\begin{equation}\label{eq:S2}
\mathscr{S}_a^2\eqdef  \bigcup_{c=1}^{a-1}\bigcup_{s=0}^K\left\{Q_s^3(c,a),Q_s^4(c,a),Q_s^5(c,a)\right\}.
\end{equation}
Thus, recalling~\eqref{eq:def B}, the sets $\mathscr{S}_1,\ldots,\mathscr{S}_n$ form a partition of $\mathscr{B}$ and $a\in \mathscr{S}_a$ for every $a\in \n$. For every $\ell\in \{0,\ldots,L\}$ let $a(\ell)$ be the unique element of $\n$ for which $u_{\ell}\in \mathscr{S}_{a(\ell)}$. Then $a(0)=i$ and $a(L)=j$. The left hand side of~\eqref{eq:goal with sum} can be bounded from below as follows
\begin{equation}\label{eq:minimum distance}
\frac{1}{N} \sum_{\ell=1}^L \|u_\ell-u_{\ell-1}\|_2^p\ge \frac{1}{N} \sum_{\ell=1}^L \min_{\substack{u\in \mathscr{S}_{a(\ell-1)}\\ v\in \mathscr{S}_{a(\ell)}}}\|u-v\|_2^p.
\end{equation}

We shall show that
\begin{equation}\label{eq:parts are distant}
\forall\, a,b\in \n,\ \forall (u,v)\in \mathscr{S}_a\times \mathscr{S}_b,\qquad \|u-v\|_2^p\ge \frac{d_X(x_a,x_b)}{m^p}.
\end{equation}
The validity of~\eqref{eq:parts are distant} implies the required estimate~\eqref{eq:goal with sum} because, by~\eqref{eq:minimum distance}, it follows from~\eqref{eq:parts are distant} and the triangle inequality for $d_X$ that
$$
\frac{1}{N} \sum_{\ell=1}^L \|u_\ell-u_{\ell-1}\|_2^p\ge \frac{1}{N}\sum_{\ell=1}^L \frac{d_X\!\left(x_{a(\ell-1)},x_{a(\ell)}\right)}{m^p}\ge \frac{d_X\!\left(x_{i},x_{j}\right)}{m^pN}.
$$

It remains to justify~\eqref{eq:parts are distant}. Suppose that $a,b\in \n$ satisfy $a<b$ and $(u,v)\in \mathscr{S}_a\times \mathscr{S}_b$. Write $u=Q_s^t(c,d)$ and $v=Q_\sigma^\tau(\gamma,\delta)$  for some $s,\sigma\in \{0,\ldots,K\}$, $t,\tau\in \{1,\ldots,5\}$ and $c,d,\gamma,\delta\in \n$.

We shall check below, via a direct case analysis, that the absolute value of one of the three coordinates of $u-v$ is either at least $M/m$ or at least $d_X(x_a,x_b)^{1/p}/m$. Since by the definition of $M$ in~\eqref{eq:mM} we have $M\ge d_X(x_a,x_b)^{1/p}$, this assertion will imply~\eqref{eq:parts are distant}.

Suppose first that $t,\tau\in \{1,2,4,5\}$. By comparing~\eqref{eq:S1}, \eqref{eq:S2} with~\eqref{eq:Q1}, \eqref{eq:Q2}, \eqref{eq:Q3}, \eqref{eq:Q4} we see that $\langle u,e_1\rangle =Ma/m$ and $\langle v,e_1\rangle =Mb/m$. Since $b-a\ge 1$, this implies that $\langle u-v,e_1\rangle\ge M/m$, as required.

If $t=\tau=3$ then by~\eqref{eq:S2} we necessarily have $d=a$ and $\delta=b$. Hence $(c,d)\neq (\gamma,\delta)$ and therefore $|\phi(c,d)-\phi(\gamma,\delta)|\ge 1$, since $\phi$ is a bijection between $\n\times \n$ and $\{1,\ldots,n^2\}$. By~\eqref{eq:Q3} we therefore have $|\langle u-v,e_2\rangle|\ge M/m$, as required.

It remains to treat the case $t\neq \tau$ and $3\in \{t,\tau\}$. If $\{t,\tau\}\subset \{1,3,5\}$ then by contrasting~\eqref{eq:Q3} with~\eqref{eq:Q1} and~\eqref{eq:Q5} we see that the third coordinate of one of the vectors $u,v$ vanishes while the third coordinate of the other vector equals $M/m$. Therefore $|\langle u-v,e_3\rangle|\ge M/m$, as required. The only remaining case is $\{t,\tau\}\subset \{2,3,4\}$. In this case $|\langle u-v,e_2\rangle|=M|\phi(c,d)-\phi(\gamma,\delta)|/m$, by~\eqref{eq:Q3}, \eqref{eq:Q2}, \eqref{eq:Q4}. So, if $(c,d)\neq(\gamma,\delta)$ then $|\phi(c,d)-\phi(\gamma,\delta)|\ge 1$, and we are done. We may therefore assume that $c=\gamma$ and $d=\delta$. Observe that by~\eqref{eq:S2} if $\{t,\tau\}=\{3,4\}$ then $\{d,\delta\}=\{a,b\}$, which contradicts $d=\delta$. So, we also necessarily have $\{t,\tau\}=\{2,3\}$, in which case, since $a<b$, by~\eqref{eq:S1} and~\eqref{eq:S2} we see that $c=\gamma=a$ and $d=\delta=b$.  By interchanging the labels $s$ and $\sigma$ if necessary, we may assume that $u=Q^2_\sigma(a,b)$ and $v=Q^3_s(a,b)$. By~\eqref{eq:Q2} and~\eqref{eq:Q3} we therefore have
\begin{multline*}
\langle v-u,e_1\rangle= \frac{M(s(b-a)+Ka)}{mK}+\frac{(K-s)d_X(x_a,x_b)^{\frac{1}{p}}}{mK}-\frac{Ma}{m}\\
=\frac{d_X(x_a,x_b)^{\frac{1}{p}}}{m}+\frac{sM(b-a)-sd_X(x_a,x_b)^{\frac{1}{p}}}{mK}\ge \frac{d_X(x_a,x_b)^{\frac{1}{p}}}{m},
\end{multline*}
where we used the fact that by~\eqref{eq:mM} we have $M\ge d_X(x_a,x_b)^{1/p}$, and that $b-a\ge 1$. This concludes the verification of the remaining case of~\eqref{eq:parts are distant}, and hence the proof of Lemma~\ref{lem:dist f} is complete.
\end{proof}

\section{Sharpness of Theorem~\ref{thm:main snowflake}}\label{sec:sharpness}

The results of this section rely crucially on K. Ball's notion~\cite{Bal92} of Markov type. We shall start by briefly recalling the relevant background on this important invariant of metric spaces, including variants and notation from~\cite{Nao14} that will be used below. Let $\{Z_t\}_{t=0}^\infty$ be a Markov chain on the state space $\n$ with transition probabilities $a_{ij}=\Pr\left[Z_{t+1}=j|Z_t=i\right]$ for every $i,j\in \n$. $\{Z_t\}_{t=0}^\infty$ is said to be stationary if $\pi_i=\Pr\left[Z_t=i\right]$ does not depend on $t\in \n$ and it is said to be  reversible if $\pi_i a_{ij}=\pi_j a_{ji}$ for every $i,j\in \n$.

Let $\{Z_t'\}_{t=0}^\infty$ be the Markov chain that starts at $Z_0$ and then evolves independently of $\{Z_t\}_{t=0}^\infty$ with the same transition probabilities. Thus $Z_0'=Z_0$ and conditioned on $Z_0$ the random variables $Z_t$ and $Z_t'$ are independent and identically distributed. We note for future use that if $\{Z_t\}_{t=0}^\infty$ as above is stationary and reversible then for every symmetric function $\psi:\n\times\n\to \R$ and every $t\in \N$ we have
\begin{equation}\label{eq:psi identity}
\E\big[\psi(Z_t,Z_t')\big]=\E\big[\psi(Z_{2t},Z_0)\big].
\end{equation}
This is a consequence of the observation that, by stationarity and revesibility, conditioned on the random variable $Z_t$ the random variables $Z_0$ and $Z_{2t}$ are independent and identically distributed. Denoting $A=(a_{ij})\in M_n(\R)$, the validity of~\eqref{eq:psi identity} can be alternatively checked directly as follows.
\begin{multline}\label{eq: check psi}
\E\left[\psi(Z_t,Z_t')\right]=\E\big[\E\left[\psi(Z_t,Z_t')|Z_0\right]\big]=\sum_{i=1}^n \sum_{j=1}^n\sum_{k=1}^n\pi_i A^t_{ij}A^t_{ik} \psi(j,k)\\\stackrel{(\star)}{=}
\sum_{j=1}^n\sum_{k=1}^n\pi_j \bigg(\sum_{i=1}^n A^t_{ji}A^t_{ik}\bigg) \psi(j,k)=\sum_{j=1}^n\sum_{k=1}^n\pi_j A^{2t}_{jk} \psi(j,k),
\end{multline}
where $(\star)$ uses the reversibility of the Markov chain $\{Z_t\}_{t=0}^\infty$  through the validity of $\pi_iA^t_{ij}=\pi_jA^t_{ji}$ for every $i,j\in \n$. The final term in~\eqref{eq: check psi} equals the right hand side of~\eqref{eq:psi identity}, as required.

Given $p\in [1,\infty)$, a metric space $(X,d_X)$ and $m\in \N$, the Markov type $p$ constant of $(X,d_X)$ at time $m$, denoted $M_p(X,d_X;m)$ (or simply $M_p(X;m)$ if the metric is clear from the context) is defined to be the infimum over those $M\in (0,\infty)$ such that for every $n\in \N$, every stationary reversible Markov chain $\{Z_t\}_{t=0}^\infty$ with state space $\n$, and every $f:\n\to X$ we have
$$
\E\big[d_X(f(Z_m),f(Z_0))^p\big]\le M^pm\E\big[d_X(f(Z_1),f(Z_0))^p\big].
$$
Observe that by the triangle inequality we always have $$M_p(X;m)\le m^{1-\frac{1}{p}}.$$ As we shall explain below, any estimate of the form $M_p(X;m)\lesssim_X m^{\theta}$ for $\theta<1-1/p$ is a nontrivial obstruction to the embeddability of certain metric spaces into $X$, but it is especially important (e.g. for Lipschitz extension theory~\cite{Bal92}) to single out the case when $M_p(X;m)\lesssim_X 1$. Specifically, $(X,d_X)$ is said to have Markov type $p$ if
$$
M_p(X,d_X)\eqdef \sup_{m\in \N} M_p(X,d_X;m)<\infty.
$$
$M_p(X,d_X)$ is called the Markov type $p$ constant of $(X,d_X)$, and it is often denoted simply $M_p(X)$ if the metric is clear from the context.

The Markov type of many important classes of metric spaces is satisfactorily understood, though some fundamental questions remain open; see Section~4 of the survey~\cite{Nao12} and the references therein, as well as more recent progress in e.g.~\cite{DLP13}. Here we study this notion in the context of Wasserstein spaces. The link of Markov type to the nonembeddability of snowflakes is simple, originating in an idea of~\cite{LMN02}. This is the content of the following lemma.

\begin{lemma}\label{lem:markov type no snowflake}
Fix a metric space $(Y,d_Y)$, $m\in \N$, $K,p\in [1,\infty)$ and $\theta\in [0,1]$. Suppose that
\begin{equation}\label{eq:Mp theta assumption}
M_p(Y;m)\le Km^{\frac{\theta(p-1)}{p}}.
\end{equation}
Denote $n=2^{4m}$. Then there exists an $n$-point metric space $(X,d_X)$ such that
$$
\alpha\in \left[\frac{1+\theta(p-1)}{p}, 1\right] \implies c_Y(X,d_X^\alpha)\gtrsim \frac{1}{K}(\log n)^{\alpha-\frac{1+\theta(p-1)}{p}}.
$$
\end{lemma}

\begin{proof}
Take $(X,d_X)=(\{0,1\}^{4m},\|\cdot\|_1)$, i.e., $X$ is the $4m$-dimensional discrete hypercube, equipped with the Hamming metric. Thus $|X|=n$. Let $\{Z_t\}_{t=0}^\infty$ be the standard random walk on $X$, with $Z_0$ distributed uniformly over $X$. Suppose that $f:X\to Y$ satisfies
\begin{equation}\label{eq:sD for lower}
\forall\, x,y\in X,\qquad s\|x-y\|_1^\alpha\le d_Y(f(x),f(y))\le Ds\|x-y\|_1^\alpha
\end{equation}
for some $s,D\in (0,\infty)$. Our goal is to bound $D$ from below. By the definition of $M_p(Y;m)$,
\begin{equation}\label{eq:use markov type}
\E\big[d_Y(f(Z_m),f(Z_0))^p\big]\stackrel{\eqref{eq:Mp theta assumption}}{\le} K^p m^{1+\theta(p-1)}  \E\big[d_Y(f(Z_1),f(Z_0))^p\big].
\end{equation}
By the right hand inequality in~\eqref{eq:sD for lower} we have
\begin{equation}\label{eq:use upper lip}
\E\big[d_Y(f(Z_1),f(Z_0))^p\big]\\
\le D^ps^p\E\big[\|Z_1-Z_0\|_1^{\alpha p}\big]= D^p s^p.
\end{equation}
At the same time, it is simple to see (and explained explicitly in e.g.~\cite{NS02} or~\cite[Section~9.4]{Nao12}) that $\E\big[\|Z_m-Z_0\|_1^{\alpha p} \big]\ge (\eta m)^{\alpha p}$ for some universal constant $\eta\in (0,1)$. Hence,
\begin{equation}\label{eq:use lower lip}
\E\big[d_Y(f(Z_m),f(Z_0))^p\big]\stackrel{\eqref{eq:sD for lower}}{\ge} s^p \E\big[\|Z_m-Z_0\|_1^{\alpha p}\big]\gtrsim s^p (\eta m)^{\alpha p}.
\end{equation}
The only way for~\eqref{eq:use upper lip} and~\eqref{eq:use lower lip} to be compatible with~\eqref{eq:use markov type} is if
\begin{equation*}
D\gtrsim \frac{1}{K} m^{\alpha-\frac{1+\theta(p-1)}{p}}\asymp \frac{1}{K}(\log n)^{\alpha-\frac{1+\theta(p-1)}{p}}.\tag*{\qedhere}
\end{equation*}
\end{proof}

\begin{remark}
{\em In Lemma~\ref{lem:markov type no snowflake} we took the metric space $X$ to be a discrete hypercube, but similar conclusions apply to snowflakes of expander graphs and graphs with large girth~\cite{LMN02}, as well as their subsets~\cite{BLMN05} and certain discrete groups~\cite{ANP09,NP08,NP11} (see also~\cite[Section~9.4]{Nao12}). We shall not attempt to state here the wider implications of the assumption~\eqref{eq:Mp theta assumption}  to the nonembeddability of snowflakes, since the various additional conclusions follow mutatis mutandis from the same argument as above, and Lemma~\ref{lem:markov type no snowflake}  as currently stated suffices for the proof of Theorem~\ref{thm:no large snowflake}.}
\end{remark}

\begin{remark} {\em Since the proof of Lemma~\ref{lem:markov type no snowflake} applied the Markov type $p$ assumption~\eqref{eq:Mp theta assumption}  to the discrete hypercube, it would have sufficed to work here with a classical weaker bi-Lipschitz invariant due to Enflo~\cite{Enf76}, called Enflo type. Such an obstruction played a role in ruling out certain snowflake embeddings in~\cite{FS11} (in a different context), though the fact that the argument of~\cite{FS11}  could be cast in the context of Enflo type was proved only later~\cite[Proposition~5.3]{Oht09}. Here we work with Markov type rather than Enflo type because the proof below for Wasserstein spaces yields this stronger conclusion without any additional effort.}
\end{remark}

The following lemma is a variant of~\cite[Lemma~4.1]{Oht09}.

\begin{lemma}\label{lem:powers of 2}
Fix $p\in [1,\infty)$ and $\theta\in [1/p,1]$. Suppose that $(X,d_X)$ is a metric space such that for every two $X$-valued independent and identically distributed finitely supported random  variables $Z,Z'$ and every $x\in X$ we have
\begin{equation}\label{eq:theta assumption X}
\E\big[d_X(Z,Z')^p\big]\le 2^{\theta p}\E\big[d_X(Z,x)^p\big].
\end{equation}
Then for every $k\in \N$ we have
\begin{equation}\label{eq:Mtype times power of 2}
M_p(X;2^k)\le 2^{k\left(\theta-\frac{1}{p}\right)}.
\end{equation}
\end{lemma}
\begin{proof} Fix $n\in \N$, a stationary reversible Markov chain $\{Z_t\}_{t=0}^\infty$ with state space $\n$, and $f:\n\to X$. Recalling~\eqref{eq:psi identity} with $\psi(i,j)=d_X(f(i),f(j))^p$, for every $t\in \N$ we have
\begin{multline}\label{eq:use independent ineq}
\E\big[d_X(Z_{2t},Z_0)^p\big]\stackrel{\eqref{eq:psi identity}}{=}\E\big[d_X(Z_{t},Z_t')^p\big]\stackrel{\eqref{eq:theta assumption X}}{\le} 2^{\theta p} \E\big[d_X(Z_{t},Z_0)^p\big]
\\\le 2^{\theta p-1}M_p(X;t)^p\cdot 2t\E\big[d_X(Z_{1},Z_0)^p\big],
\end{multline}
where the last step of~\eqref{eq:use independent ineq} uses the definition of $M_p(X;t)$. By the definition of $M_p(X;2t)$, we have thus proved that
$$
M_p(X;2t)\le 2^{\theta-\frac{1}{p}}M_p(X;t),
$$
so~\eqref{eq:Mtype times power of 2} follows by induction on $k$.
\end{proof}

Corollary~\ref{coro:alpha minus theta} below follows from Lemma~\ref{lem:markov type no snowflake} and Lemma~\ref{lem:powers of 2}. Specifically, under the assumptions and notation of Lemma~\ref{lem:powers of 2}, use Lemma~\ref{lem:markov type no snowflake} with $m$ replaced by $2^k$ and $\theta$ replaced by $(\theta p-1)/(p-1)$.

\begin{corollary}\label{coro:alpha minus theta}
Fix $p\in [1,\infty)$ and $\theta\in [1/p,1]$. Suppose that $(X,d_X)$ is a metric space that satisfies the assumptions of Lemma~\ref{lem:powers of 2}. Then for arbitrarily large $n\in \N$ there exists an $n$-point metric space $(Y,d_Y)$ such that for every $\alpha\in [\theta,1]$ we have
$$
c_X\left(Y,d_Y^\alpha\right)\gtrsim (\log n)^{\alpha-\theta}.
$$
\end{corollary}

The link between the above discussion and embeddings of snowflakes of metrics into Wasserstein spaces is explained in the following lemma, which is a variant of~\cite[Proposition~2.10]{Stu06}.
\begin{lemma}\label{lem:p sturm}
Fix $p\in [1,\infty)$ and $\theta\in [1/p,1]$. Suppose that $(X,d_X)$ is a metric space that satisfies the assumptions of Lemma~\ref{lem:powers of 2}, i.e., inequality~\eqref{eq:theta assumption X} holds true for $X$-valued random variables. Then the same inequality holds true in the metric space $(\Pp(X),\W_p)$ as well, i.e., for every two $\Pp(X)$-valued and identically distributed finitely supported random  variables $\mathfrak{M}, \mathfrak{M'}$ and every $\mu\in \Pp(X)$,
\begin{equation*}
\E\big[\W_p(\mathfrak{M},\mathfrak{M}')^p\big]\le 2^{\theta p}\E\big[\W_p(\mathfrak{M},\mu)^p\big].
\end{equation*}
\end{lemma}

\begin{proof}
Suppose that the distribution of $\mathfrak{M}$ equals $\sum_{i=1}^n q_i\delta_{\mu_i}$ for some $\mu_1,\ldots,\mu_n\in \Pp(X)$ and $q_1,\ldots,q_n\in [0,1]$ with $\sum_{i=1}^nq_i=1$. Our goal is to show that
\begin{equation}\label{eq:curvature in coordinates}
\sum_{i=1}^n\sum_{j=1}^n q_iq_j \W_p(\mu_i,\mu_j)^p\le 2^{\theta p}\sum_{i=1}^nq_i \W_p(\mu_i,\mu)^p.
\end{equation}
The finitely supported probability measures are dense in $(\Pp(X),\W_p)$ (see~\cite{RR98,Vil03}), so it suffices to prove~\eqref{eq:curvature in coordinates} when there exists $N\in \N$ and points $x_{ik},x_k\in X$ for every $(i,k)\in \n\times \{1,\ldots, N\}$
such that we have $\mu=\frac{1}{N}\sum_{k=1}^N \delta_{x_k}$ and $\mu_i=\frac{1}{N}\sum_{k=1}^N \delta_{x_{ik}}$ for every $i\in \n$. Let $\{\sigma_i\}_{i=1}^N\subset S_N$ be permutations of $\{1,\ldots,N\}$ that induce optimal couplings of the pairs $(\mu,\mu_i)$, i.e.,
\begin{equation}\label{eq:sigmai optimal}
\forall\, i\in \n,\qquad  \W_p(\mu_i,\mu)^p=\frac{1}{N}\sum_{k=1}^N d_X(x_{i\sigma_i(k)},x_{k})^p.
\end{equation}
Since the measure $\frac{1}{N}\sum_{k=1}^N \delta_{(x_{i\sigma_i(k)},x_{j\sigma_j(k)})}$ is a coupling of $(\mu_i,\mu_j)$,
\begin{equation}\label{eq:sigmai sigmaj}
\forall\, i,j\in \n,\qquad \W_p(\mu_i,\mu_j)^p\le \frac{1}{N}\sum_{k=1}^N d_X(x_{i\sigma_i(k)},x_{j\sigma_j(k)})^p.
\end{equation}
Consequently,
\begin{multline*}
\sum_{i=1}^n\sum_{j=1}^n q_iq_j \W_p(\mu_i,\mu_j)^p\stackrel{\eqref{eq:sigmai sigmaj}}{\le} \frac{1}{N}\sum_{k=1}^N\sum_{i=1}^n\sum_{j=1}^n q_iq_jd_X(x_{i\sigma_i(k)},x_{j\sigma_j(k)})^p\\\stackrel{\eqref{eq:theta assumption X}}{\le} \frac{2^{\theta p}}{N}\sum_{k=1}^N\sum_{i=1}^n\sum_{j=1}^n q_iq_j d_X(x_{i\sigma_i(k)},x_{k})^p\stackrel{\eqref{eq:sigmai optimal}}{=} 2^{\theta p}\sum_{i=1}^nq_i \W_p(\mu_i,\mu)^p.\tag*{\qedhere}
\end{multline*}
\end{proof}

\begin{proof}[Proof of Theorem~\ref{thm:no large snowflake}] Let $(\Omega,\mu)$ be a probability space. For $p\in [1,\infty]$ define $T:L_p(\mu)\to L_p(\mu\times \mu)$ by $Tf(x,y)=f(x)-f(y)$. Then clearly
$\|T\|_{L_p(\mu)\to L_p(\mu\times \mu)}\le 2$ for $p\in \{1,\infty\}$ and
$$
\forall\, f\in L_2(\mu),\qquad \|Tf\|_{L_2(\mu\times \mu)}^2=2\|f\|_{L_2(\mu)}^2-2\Big(\int_\Omega f\d\mu\Big)^2\le 2\|f\|_{L_2(\mu)}^2.
$$
Or $\|T\|_{L_2(\mu)\to L_2(\mu\times \mu)}\le \sqrt{2}$. So, by the Riesz--Thorin  theorem (e.g.~\cite{Gar07}),
\begin{equation}\label{eq:p<2}
p\in [1,2]\implies \|T\|_{L_p(\mu)\to L_p(\mu\times \mu)} \le 2^{\frac{1}{p}},
\end{equation}
and
\begin{equation}\label{eq:p>2}
p\in [2,\infty]\implies \|T\|_{L_p(\mu)\to L_p(\mu\times \mu)} \le 2^{1-\frac{1}{p}}.
\end{equation}
 Switching to probabilistic terminology, the estimates~\eqref{eq:p<2} and~\eqref{eq:p>2} say that if $Z,Z'$ are i.i.d. random variables then $\E\big[|Z-Z'|^p\big]\le 2\E\big[|Z|^p\big]$ when $p\in [1,2]$ and  $\E\big[|Z-Z'|^p\big]\le 2^{p-1}\E\big[|Z|^p\big]$ when $p\in [2,\infty)$. By applying this to the random variables $Z-a,Z'-a$ for every $a\in \R$, we deduce that the real line (with its usual metric) satisfies~\eqref{eq:theta assumption X} with
\begin{equation}\label{eq:thetap}
\theta=\theta_p\eqdef \max\left\{\frac{1}{p},1-\frac{1}{p}\right\}.
\end{equation}
Invoking this statement coordinate-wise shows that $\ell_p^3=(\R^3,\|\cdot\|_p)$ satisfies~\eqref{eq:theta assumption X} with $\theta=\theta_p$. Lemma~\ref{lem:p sturm} therefore implies that $(\Pp(\ell_p^3),\W_p)$ also satisfies~\eqref{eq:theta assumption X} with $\theta=\theta_p$. Hence, by Corollary~\ref{coro:alpha minus theta} for arbitrarily large $n\in \N$ there exists an $n$-point metric space $(Y,d_Y)$ such that for every $\alpha \in (\theta_p,1]$,
\begin{equation*}\label{eq:put together distortion}
c_{(\Pp(\ell_p^3),\W_p)}(Y,d_Y^\alpha)\gtrsim (\log n)^{\alpha-\theta_p}=\left\{\begin{array}{ll}(\log n)^{\alpha-\frac{1}{p}}&\mathrm{if}\ p\in (1,2],\\
(\log n)^{\alpha+\frac{1}{p}-1}&\mathrm{if}\ p\in (2,\infty).\end{array}\right.
\end{equation*}
Since the $\ell_p$ norm on $\R^3$ is $\sqrt{3}$-equivalent to the $\ell_2$ norm on $\R^3$,
$$
c_{(\Pp(\ell_p^3),\W_p)}(Y,d_Y^\alpha)\asymp c_{(\Pp(\ell_2^3),\W_p)}(Y,d_Y^\alpha),
$$
thus completing the proof of Theorem~\ref{thm:no large snowflake}.
\end{proof}

\begin{remark} {\em In the proof of Theorem~\ref{thm:no large snowflake} we chose to check the validity of~\eqref{eq:theta assumption X}  with $\theta=\theta_p$ given in~\eqref{eq:thetap} using an interpolation argument since it is very short. But, there are different proofs of this fact: when $p\in [1,2)$ one could start from the trivial case $p=2$, and then pass to general $p\in [1,2)$ by invoking the classical fact~\cite{Sch38} that the metric space $(\R,|x-y|^{p/2})$ admits an isometric embedding into Hilbert space. Alternatively, in~\cite[Lemma~3]{Nao01} this is proved via a direct computation.}
\end{remark}

\begin{question}\label{Q:p>2}
As discussed in the Introduction, it seems plausible that Theorem~\ref{thm:main snowflake} and Theorem~\ref{thm:no large snowflake} are not sharp when $p\in (2,\infty)$. Specifically, we conjecture that there exist $D_p\in [1,\infty)$ such that for every finite metric space $(X,d_X)$ we have
\begin{equation}\label{eq:Dp p>2}
c_{\Pp(\R^3)}\Big(X,\sqrt{d_X}\Big)\le D_p.
\end{equation}
Perhaps~\eqref{eq:Dp p>2} even holds true with $D_p=1$.  As discussed in Remark~\ref{rem:analogy with lp}, since $L_2$ admits an isometric embedding into $L_p$ (see e.g.~\cite{Woj91}), the perceived analogy between Wasserstein $p$ spaces and $L_p$ spaces makes it natural to ask whether or not $(\mathscr{P}_{\!2}(\R^3),\W_2)$ admits a bi-Lipschitz embedding into $(\Pp(\R^3),\W_p)$. If the answer to this question were positive then~\eqref{eq:Dp p>2} would hold true by virtue of the case $p=2$ of Theorem~\ref{thm:main snowflake}. We also conjecture that the lower bound of Theorem~\ref{thm:no large snowflake} could be improved when $p>2$ to state that for arbitrarily large $n\in \N$ there exists an $n$-point metric space $(Y,d_Y)$ such that for every $\alpha\in (1/2,1]$,
\begin{equation}\label{eq:lower p>2 conj}
c_{(\Pp(\R^3),\W_p)}(Y,d_Y^\alpha)\gtrsim_p (\log n)^{\alpha-\frac12}.
\end{equation}
It was shown in~\cite{NPSS06} that $L_p$ has Markov type $2$ for every $p\in (2,\infty)$. We therefore ask whether or not $(\Pp(\R^3),\W_p)$ has Markov type $2$ for every $p\in (2,\infty)$. A positive answer to this question would imply that the lower bound~\eqref{eq:lower p>2 conj} is indeed achievable. For this purpose it would also suffice to show that for every $p\in (2,\infty)$ and $k\in \N$ we have
\begin{equation}\label{eq:markov type 2 power p wasserstein}
M_p((\Pp(\R^3),\W_p);2^k)\lesssim_p 2^{k\left(\frac12-\frac{1}{p}\right)}.
\end{equation}
Proving~\eqref{eq:markov type 2 power p wasserstein} may be easier than proving that $M_2(\Pp(\R^3),\W_p)<\infty$, since the former involves arguing about the $p$th powers of Wasserstein~$p$ distances while the latter involves arguing about Wasserstein $p$ distances squared. Note that $M_p(L_p;m)\lesssim \sqrt{p} m^{1/2-1/p}$ by~\cite{NPSS06} (see also~\cite[Theorem~4.3]{Nao14}), so the $L_p$-version of~\eqref{eq:markov type 2 power p wasserstein} is indeed valid.
\end{question}

We end this section by showing how Lemma~\ref{lem:powers of 2} implies bounds on the Markov type $p$ constant $M_p(X;t)$ for any time $t\in \N$, and not only when $t=2^k$ for some $k\in \N$ as in~\eqref{eq:Mtype times power of 2}. For the purpose of proving Theorem~\ref{thm:main snowflake}, Lemma~\ref{lem:powers of 2} suffices as stated, so the ensuing discussion is included for completeness, and could be skipped by those who are interested only in the proof of Theorem~\ref{thm:no large snowflake}.


The case $p=2$ and $\theta=1/2$ of Lemma~\ref{lem:powers of 2} corresponds to proving that metric spaces that are nonnegatively curved in the sense of Alexandrov have Markov type $2$: this was established by Ohta in~\cite{Oht09}, whose work inspired the arguments that were presented above. Specifically, Ohta showed in~\cite{Oht09} how to pass from~\eqref{eq:Mtype times power of 2} with $p=2$ and $\theta=1/2$ (i.e., $M_2(X,2^k)\le 1$ for every $k\in \N$) to $M_2(X)\le \sqrt{6}=2.449...$, and he also included in~\cite{Oht09} an argument of Naor and Peres that improves this to $M_2(X)\le 1+\sqrt{2}=2.414..$. Below we further refine the latter argument, yielding the best known estimate on the Markov type $2$ constant of Alexandrov spaces of nonnegative curvature; see~\eqref{eq:our M2 estimate} below. This constant is of interest since it was shown in~\cite{OP09} that if $(X,d_X)$ is a geodesic metric space with $M_2(X)=1$ then $X$ is nonnegatively curved in the sense of Alexandrov. It is plausible that, conversely, $M_2(X)=1$ if $X$ is nonnegatively curved in the sense of Alexandrov, but, as noted in~\cite{OP09}, this seems to be unknown even for the circle  $X=S^1$.

For every $\theta\in (0,1]$ define $\phi_\theta:[0,1]\to \R$ by
\begin{equation}\label{eq:def phi theta}
\forall\, s\in [0,1],\qquad \phi_\theta(s)\eqdef s^\theta-(1-s)^\theta.
\end{equation}
Then $\phi_\theta([0,1])=[-1,1]$ and since $\phi_\theta'(s)=\theta s^{\theta-1}+\theta(1-s)^{\theta-1}>0$, the inverse $\phi_\theta^{-1}$ is well-defined and increasing on $[-1,1]$. The following elementary numerical lemma will be used later.

\begin{lemma}\label{lem:def c theta} For all $\theta\in (0,1]$ there is a unique $c(\theta)\in (1,\infty)$ satisfying
\begin{equation}\label{eq:fixed point equation}
c(\theta)=\frac{c(\theta)\phi_\theta^{-1}\left(\frac{2^\theta-1}{c(\theta)}\right)^\theta+1}{\left(\phi_\theta^{-1}\left(\frac{2^\theta-1}{c(\theta)}\right)+1\right)^\theta}.
\end{equation}
\end{lemma}

\begin{proof} The identity~\eqref{eq:fixed point equation} is equivalent to $h_\theta(c(\theta))=1$, where for every $s>0$ and $c\in [1,\infty)$ we set
$$
\psi_\theta(s)\eqdef (s+1)^\theta-s^\theta \qquad \mathrm{and}\qquad h_\theta(c)\eqdef c\psi_\theta\left(\phi_\theta^{-1}\left(\frac{2^\theta-1}{c}\right)\right).
$$
Observe that because $\theta\in (0,1]$ we have $\psi_\theta(s)<1$ for every $s>0$. Hence $h_\theta(c)<c$ for every $c\in (0,\infty)$, and in particular $h_\theta(1)<1$. Moreover, $\phi_\theta^{-1}(0)=1/2$, so that
$$
\lim_{c\to\infty} \psi_\theta\left(\phi_\theta^{-1}\left(\frac{2^\theta-1}{c}\right)\right)=\psi_\theta\left(\frac12\right)=\frac{3^\theta}{2^\theta}-\frac{1}{2^\theta}>0.
$$
Hence $\lim_{c\to\infty} h_\theta(c)=\infty$. It follows by continuity that there exists $c\in (0,\infty)$ such that $h_\theta(c)=1$. To prove the uniqueness of such $c>1$, it suffice to show that $h_\theta$ is increasing on $(0,\infty)$. Now,
\begin{align*}
h_\theta'(c)=\psi_\theta\left(\phi_\theta^{-1}\left(\frac{2^\theta-1}{c}\right)\right)-\frac{2^\theta-1}{c}\cdot
\frac{\psi_\theta'\left(\phi_\theta^{-1}\left(\frac{2^\theta-1}{c}\right)\right)}{\phi_\theta'\left(\phi_\theta^{-1}\left(\frac{2^\theta-1}{c}\right)\right)}
=\frac{\phi_\theta'(y)\psi_\theta(y)-\phi_\theta(y)\psi_\theta'(y)}{\phi_\theta'(y)},
\end{align*}
where we write $y=\phi_\theta^{-1}((2^\theta-1)/c)$. Since $\phi_\theta$ is increasing, it therefore suffices to show that $\phi_\theta'(y)\psi_\theta(y)-\phi_\theta(y)\psi_\theta'(y)>0$ for all $y\in (0,1)$.
One directly computes that
$$
\phi_\theta'(y)\psi_\theta(y)-\phi_\theta(y)\psi_\theta'(y)=\theta\cdot \frac{2y^{1-\theta}+(1-y)^{1-\theta}-(1+y)^{1-\theta}}{y^{1-\theta}(1-y^2)^{1-\theta}}.
$$
It remains to note that by the subadditivity of $t\mapsto t^{1-\theta}$ we have
\begin{equation*}
(1+y)^{1-\theta}\le (1-y)^{1-\theta}+(2y)^{1-\theta}\le (1-y)^{1-\theta}+2y^{1-\theta}.\tag*{\qedhere}
\end{equation*}
\end{proof}

\begin{lemma}\label{lem:pass to all t}
Fix $p\in [1,\infty)$ and $\theta\in [1/p,1]$. Suppose that $(X,d_X)$ is a metric space that satisfies the assumptions of Lemma~\ref{lem:powers of 2}, i.e., inequality~\eqref{eq:theta assumption X} holds true for $X$-valued random variables. Then
$$
\forall\, t\in \N,\qquad M_p(X;t)\le c(\theta) t^{\theta-\frac{1}{p}},
$$
where $c(\theta)$ is from Lemma~\ref{lem:def c theta}. Thus, if $\theta=1/p$ then $X$ has Markov type $p$ with $M_p(X)\le c(1/p)$.
\end{lemma}

 Because, by the Lang--Schroeder--Sturm inequality~\eqref{eq:LSS}, Alexandrov spaces of nonnegative curvature satisfy the assumption of Lemma~\ref{lem:powers of 2} with $p=2$ and $\theta=1/2$, we have the following corollary. Note that $c(1/2)$ can be computed explicitly   by solving the equation~\eqref{eq:fixed point equation}.

\begin{corollary}\label{eq:Mtype 2 const}
Suppose that $(X,d_X)$ is nonnegatively curved in the sense of Alexandrov. Then  the Markov type $2$ constant of $X$ satisfies
\begin{equation}\label{eq:our M2 estimate}
M_2(X)\le c\left(\frac12\right)= \sqrt{1+\sqrt{2}+\sqrt{4\sqrt{2}-1}}=2.08...
\end{equation}
\end{corollary}

\begin{proof}[Proof of Lemma~\ref{lem:pass to all t}] We claim that the number $c(\theta)$ of Lemma~\ref{lem:def c theta} satisfies
\begin{equation}\label{for induction verison c theta}
\sup_{s\in [0,1]}\frac{\min\left\{1+c(\theta) s^\theta,2^\theta+c(\theta)(1-s)^\theta\right\}}{(1+s)^\theta}=c(\theta).
\end{equation}
Indeed, observe that the function $s\mapsto (1+c(\theta) s^\theta)/(1+s)^\theta$ is increasing on $[0,1]$ because one directly computes that its derivative equals $\theta(c(\theta)-s^{1-\theta})/(s^{1-\theta}(1+s)^{1+\theta})$, and by Lemma~\ref{lem:def c theta} we have $c(\theta)>1$ (recall also that $0<\theta\le 1$). Since the function $s\mapsto (2^\theta+c(\theta)(1-s)^\theta)/(1+s)^\theta$ is decreasing on $[0,1]$, it follows that the supremum that appears in the left hand side of~\eqref{for induction verison c theta} is attained when $1+c(\theta) s^\theta=2^\theta+c(\theta)(1-s)^\theta$, or equivalently when $\phi_\theta(s)=(2^\theta-1)/c(\theta)$, where we recall~\eqref{eq:def phi theta}. Thus $s=\phi_\theta^{-1}((2^\theta-1)/c(\theta))$ and therefore~\eqref{for induction verison c theta} is equivalent to~\eqref{eq:fixed point equation}.

Fix $n\in \N$, a stationary reversible Markov chain $\{Z_t\}_{t=0}^\infty$ on $\n$, and $f:\n\to X$. For simplicity of notation write $U_t=f(Z_t)$.  We shall prove by induction on $t\in \N$ that
\begin{equation}\label{eq: t induction}
\E\big[d_X(U_t,U_0)^p\big]\le c(\theta)^pt^{\theta p} \E\big[d_X(U_1,U_0)^p\big].
\end{equation}

Lemma~\ref{lem:powers of 2} shows that~\eqref{eq: t induction} holds true if $t=2^k$ for some $k\in \N\cup\{0\}$ (since $c(\theta)>1$). So, suppose that $t=(1+s)2^k$ for some $s\in (0,1)$ and $k\in \N\cup \{0\}$. The triangle inequality in $L_p$, combined with the  stationarity of the Markov chain, implies that
\begin{align}\label{eq:first triangle}
\nonumber\left(\E\big[d_X(U_t,U_0)^p\big]\right)^{\frac{1}{p}}&\le \left(\E\big[d_X(U_t,U_{2^k})^p\big]\right)^{\frac{1}{p}}+\left(\E\big[d_X(U_{2^k},U_0)^p\big]\right)^{\frac{1}{p}}\\
&=\left(\E\big[d_X(U_{s2^k},U_0)^p\big]\right)^{\frac{1}{p}}+\left(\E\big[d_X(U_{2^k},U_0)^p\big]\right)^{\frac{1}{p}},
\end{align}
and
\begin{align}\label{eq:second triangle}
\nonumber\left(\E\big[d_X(U_t,U_0)^p\big]\right)^{\frac{1}{p}}&\le \left(\E\big[d_X(U_t,U_{2^{k+1}})^p\big]\right)^{\frac{1}{p}}+\left(\E\big[d_X(U_{2^{k+1}},U_0)^p\big]\right)^{\frac{1}{p}}\\
&=\left(\E\big[d_X(U_{(1-s)2^k},U_0)^p\big]\right)^{\frac{1}{p}}+\left(\E\big[d_X(U_{2^{k+1}},U_0)^p\big]\right)^{\frac{1}{p}}.
\end{align}
By combining~\eqref{eq:first triangle} and~\eqref{eq:second triangle} with Lemma~\ref{lem:powers of 2} and the inductive hypothesis~\eqref{eq: t induction}, we see that
\begin{align*}
\frac{\left(\E\big[d_X(U_t,U_0)^p\big]\right)^{\frac{1}{p}}}{\left(\E\big[d_X(U_1,U_0)^p\big]\right)^{\frac{1}{p}}}&\le 2^{k\theta}\min\left\{c(\theta)s^\theta+1,c(\theta)(1-s)^\theta+2^\theta\right\}
&\stackrel{\eqref{for induction verison c theta}}{\le} 2^{k\theta}c(\theta)(1+s)^\theta=c(\theta)t^\theta.\tag*{\qedhere}
\end{align*}
\end{proof}

\section{Proof of Proposition~\ref{prop:duality}}\label{sec:duality}

Here we justify the validity of Proposition~\ref{prop:duality} that was stated in the Introduction, thus explaining why we are focusing on quadratic inequalities in the context of the quest for intrinsic characterizations of those metric spaces that admit a bi-Lipschitz embedding into some Alexandrov space that is either nonnegatively or nonpositively curved. The argument below is inspired by the proof of Proposition~15.5.2 in~\cite{Mat02}.

\begin{proof}[Proof of Proposition~\ref{prop:duality}]
If $c_Y(X)\le D$ for some $(Y,d_Y)\in \mathscr{F}$ then it follows immediately that if $A,B\in M_n(\R)$ have nonnegative entries and $(Y,d_Y)$ satisfies the $(A,B)$-quadratic metric inequality then $(X,d_X)$ satisfies the $(A,D^2B)$-quadratic metric inequality. The nontrivial direction here is the converse, i.e., suppose that $(X,d_X)$ satisfies the $(A,D^2B)$-quadratic metric inequality for every two $n$ by $n$ matrices $A,B\in M_n(\R)$ with nonnegative entries such that every $(Z,d_Z)\in \mathscr{F}$ satisfies the $(A,B)$-quadratic metric inequality. The goal is to deduce from this that there exists $(Y,d_Y)\in \mathscr{F}$ for which $c_Y(X)\le D$.

 Let $\mathscr{K}\subset M_n(\R)$ be the set of all $n$ by $n$ matrices $C=(c_{ij})$ for which there exists $(Z,d_Z)\in \mathscr{F}$  and $z_1,\ldots,z_n\in Z$ such that $c_{ij}=d_Z(z_i,z_j)^2$ for every $i,j\in \n$. Since $\mathscr{F}$ is closed under dilation, we have $[0,\infty)\mathscr{K}\subset \mathscr{K}$. Since $\mathscr{F}$ is closed under Pythagorean sums, we have $\mathscr{K}+\mathscr{K}\subset \mathscr{K}$. Thus $\mathscr{K}$ is a convex cone.

Write $X=\{x_1,\ldots,x_n\}$. Fix $\e\in (0,1)$ and suppose for the sake of obtaining a contradiction that there does not exist an embedding of $X$ into any member of $\mathscr{F}$ with distortion less than $D+\e$. Let $\mathscr{L}\subset M_n(\R)$ be the set of all $n$ by $n$ symmetric matrices $C=(c_{ij})$ for which there exists $s\in (0,\infty)$ such that $sd_X(i,j)^2\le c_{ij}\le (D+\e)^2sd_X(i,j)^2$ for every $i,j\in \n$. Our contrapositive assumption means that $\mathscr{K}\cap \mathscr{L}=\emptyset$. Since $\mathscr{K}$ and $\mathscr{L}\cup\{0\}$ are both cones, the separation theorem  now implies that there exists a symmetric matrix $H=(h_{ij})\in M_n(\R)$, not all of whose off-diagonal entries vanish, such that
\begin{equation}\label{eq:separation}
\inf_{C\in \mathscr{L}} \sum_{i=1}^n\sum_{j=1}^n h_{ij} c_{ij}\ge 0\ge \sup_{C\in \mathscr{K}} \sum_{i=1}^n\sum_{j=1}^n h_{ij} c_{ij}.
\end{equation}

Define $A,B\in M_n(\R)$ by setting for every $i,j\in \n$,
$$
a_{ij}\eqdef \left\{\begin{array}{ll} h_{ij} &\mathrm{if}\ h_{ij}\ge 0,\\ 0 & \mathrm{if}\ h_{ij}< 0,
\end{array}\right.\qquad\mathrm{and}\qquad b_{ij}\eqdef \left\{\begin{array}{ll} |h_{ij}| &\mathrm{if}\ h_{ij}< 0,\\ 0 & \mathrm{if}\ h_{ij}\ge 0,
\end{array}\right.
$$
The right hand inequality in~\eqref{eq:separation}, combined with the definition of $\mathscr{K}$, implies that every $(Y,d_Y)\in \mathscr{F}$  satisfies the $(A,B)$-quadratic metric inequality. By our assumption on $X$, this implies that
\begin{equation}\label{eq:use AD2B}
 \sum_{i=1}^n \sum_{j=1}^n a_{ij} d_X(x_i,x_j)^2\le D^2 \sum_{i=1}^n \sum_{j=1}^n b_{ij} d_X(x_i,x_j)^2<(D+\e)^2\sum_{i=1}^n \sum_{j=1}^n b_{ij} d_X(x_i,x_j)^2,
\end{equation}
where we used the fact that not all the off-diagonal entries of $H$ vanish, so all the sums appearing~\eqref{eq:use AD2B} are positive. Consequently, if we set
$$
\forall\, i,j\in \n,\qquad c_{ij}\eqdef \left\{\begin{array}{ll} (D+\e)^2d_X(x_i,x_j)^2 &\mathrm{if}\ h_{ij}<0\\ d_X(x_i,x_j)^2 & \mathrm{if}\ h_{ij}\ge 0,
\end{array}\right.
$$
then $C=(c_{ij})\in \mathscr{L}$ and by~\eqref{eq:use AD2B} we have $\sum_{i=1}^n\sum_{j=1}^n h_{ij} c_{ij}<0$. This contradicts the left hand inequality in~\eqref{eq:separation}.
\end{proof}

\section{Subsets of Hadamard spaces}\label{sec:hadamard}

As we discussed in the introduction, it is a major open problem to characterize those finite metric spaces that admit a bi-Lipschitz (or even isometric) embedding into some Hadamard space. By Proposition~\ref{prop:duality}, this amounts to understanding those quadratic metric inequalities that hold true in any Hadamard space. In this section we shall derive potential families of such inequalities.

An equivalent characterization of when a metric space $(X,d_X)$ is a Hadamard space is the requirement that there exists a mapping $\mathfrak{B}$  that assigns a point $\mathfrak{B}(\mu)\in X$ to every finitely supported probability measure $\mu$ on $X$ with the property that $\mathfrak{B}(\delta_x)=x$ for every $x\in X$ (i.e., $\mathfrak{B}$ is a barycenter map) and every finitely supported probability measure $\mu$ on $X$ satisfies the following inequality for every $x\in X$.
\begin{equation}\label{eq:B axiom}
d_X(x,\mathfrak{B}(\mu))^2+\int_X d_X(\mathfrak{B}(\mu),y)^2\d \mu(y)\le \int_X d_X(x,y)^2\d \mu(y).
\end{equation}
For the proof that $(X,d_X)$ is a Hadamard space if and only if it satisfies~\eqref{eq:B axiom}, see e.g. Lemma~4.4. and Theorem~4.9 in~\cite{Stu03}. One could extend the validity of~\eqref{eq:B axiom} to probability measures that are not necessarily finitely supported, but this will be irrelevant for our purposes.

Lemma~\ref{lem:harmonic} below yields a general recipe for producing quadratic metric inequalities that hold true in any Hadamard space.

\begin{lemma}\label{lem:harmonic} Fix $n\in \N$ and $p_1,\ldots,p_n,q_1,\ldots,q_n\in (0,1)$ such that $\sum_{i=1}^np_i=\sum_{j=1}^nq_j=1$. Suppose that $A=(a_{ij}), B=(b_{ij})\in M_n(\R)$ are $n$ by $n$ matrices with nonnegative entries that satisfy
\begin{equation}\label{eq:AB assumption}
\forall\,i,j\in \n,\qquad \sum_{k=1}^n a_{ik}+\sum_{k=1}^n b_{kj}=p_i+q_j.
\end{equation}
If $(X,d_X)$ is a Hadamard space then for every $x_1,\ldots,x_n\in X$ we have
\begin{equation}\label{eq:harmonic mean}
\sum_{i=1}^n\sum_{j=1}^n \frac{a_{ij}b_{ij}}{a_{ij}+b_{ij}}d_X(x_i,x_j)^2\le \sum_{i=1}^n\sum_{j=1}^n p_iq_j d_X(x_i,x_j)^2.
\end{equation}
\end{lemma}

\begin{proof} Writing $z=\mathfrak{B}\Big(\sum_{i=1}^n p_i\delta_{x_i}\Big)$, by~\eqref{eq:B axiom} for every $j\in \n$ we have
\begin{equation}\label{eq:barycenter distance from j}
d_X(x_j,z)^2+\sum_{i=1}^n p_id_X(x_i,z)^2 \le \sum_{i=1}^n p_i d_X(x_i,x_j)^2.
\end{equation}
By multiplying~\eqref{eq:barycenter distance from j} by $q_j$ and summing over $j\in \n$ we get
\begin{equation}\label{eq:summed over qj}
\sum_{j=1}^n q_j d_X(x_j,z)^2+\sum_{i=1}^n p_id_X(x_i,z)^2 \le \sum_{i=1}^n\sum_{j=1}^n p_iq_j d_X(x_i,x_j)^2.
\end{equation}
Hence,
\begin{align}
 \sum_{j=1}^n q_j d_X(x_j,z)^2+\sum_{i=1}^n p_id_X(x_i,z)^2  \label{eq:use AB identity} &=\sum_{i=1}^n\sum_{j=1}^n \left(a_{ij}d_X(x_i,z)^2+b_{ij}d_X(x_j,z)^2\right)\\&\ge
\sum_{i=1}^n\sum_{j=1}^n \frac{a_{ij}b_{ij}}{a_{ij}+b_{ij}}d_X(x_i,x_j)^2,\label{eq:use convesity square}
\end{align}
where in~\eqref{eq:use AB identity} we used~\eqref{eq:AB assumption}, and~\eqref{eq:use convesity square} holds true because    $d_X(x_i,z)+d_X(x_j,z)\ge d_X(x_i,x_j)$ for every $i,j\in \n$, and for every $s,t,\gamma\in [0,\infty)$ we have (by e.g.~Cauchy–-Schwarz),
\begin{equation}\label{eq:sharp st}
\min_{\substack{\alpha,\beta\in [0,\infty)\\ \alpha+\beta\ge \gamma}}\left(s\alpha^2+t\beta^2\right)= \frac{st\gamma^2}{s+t}.
\end{equation}
The desired estimate~\eqref{eq:harmonic mean} is a combination of~\eqref{eq:summed over qj} and~\eqref{eq:use convesity square}.
\end{proof}

The proof of Lemma~\ref{lem:harmonic} is a systematic way to exploit the existence of barycenters in order to deduce quadratic metric inequalities, under the crucial constraint that the final inequality is allowed to involve only distances within the subset  $\{x_1,\ldots,x_n\}\subset X$. The barycentric inequality~\eqref{eq:B axiom} is used in~\eqref{eq:barycenter distance from j}, but one must then remove all reference to the  auxiliary  point $z$ since it need not be part of the given subset $\{x_1,\ldots,x_n\}$. It is natural to do so by incorporating the triangle inequality $d_X(x_i,z)+d_X(x_j,z)\ge d_X(x_i,x_j)$ for some $i,j\in \n$. This inequality is distributed among the possible pairs $i,j\in \n$ through a general choice of re-weighting matrices $A,B$, with the final step in~\eqref{eq:use convesity square} being sharp due to~\eqref{eq:sharp st}. A more general  scheme along these lines will be described in Section~\ref{sec:hierarchy} below, but (an iterative applications of) the above simple scheme is already powerful, and in fact we do not know whether or not it yields a characterization of subsets of Hadamard spaces; see Question~\ref{eq:is this all?} below.

A notable special case of Lemma~\ref{lem:harmonic} is when $(p_1,\ldots,p_n)=(q_1,\ldots,q_n)$ and there exists a permutation $\pi\in S_n$ such that $a_{i\pi(i)}=b_{\pi^{-1}(i)i}=p_i$  for every $i\in \n$, while all the other entries of the matrices $A$ and $B$ vanish. In this case one arrives at the following useful inequality.

\begin{corollary}\label{coro:permutation} Suppose that $(X,d_X)$ is a Hadamard space. Then for every $n\in \N$, every $x_1,\ldots,x_n$, every $p_1,\ldots,p_n\in [0,1]$ with $\sum_{j=1}^np_j=1$ and every permutation $\pi\in S_n$ we have
\begin{equation}\label{eq:permutation version}
\sum_{i=1}^n \frac{p_ip_{\pi(i)}}{p_i+p_{\pi(i)}}d_X(x_i,x_{\pi(i)})^2\\\le \sum_{i=1}^n\sum_{j=1}^np_ip_j d_X(x_i,x_j)^2.
\end{equation}
\end{corollary}
When $n=4$ and $\pi=(1,3)(2,4)$, Corollary~\ref{coro:permutation} becomes

\begin{corollary}\label{coro:three degrees of freedom} Suppose that $(X,d_X)$ be a Hadamard space and fix $x_1,x_2,x_3,x_4\in X$. Then for every $p_1,p_2,p_3,p_4\in [0,\infty)$ we have
\begin{multline}\label{eq:weighted quadruple}
p_1p_2 d_X(x_1,x_2)^2+p_2p_3d_X(x_2,x_3)^2+p_3p_4d_X(x_3,x_4)^2+p_4p_1d_X(x_4,x_1)^2\\\ge
\frac{p_1p_3(p_2+p_4)}{p_1+p_3}d_X(x_1,x_3)^2+\frac{p_2p_4(p_1+p_3)}{p_2+p_4}d_X(x_2,x_4)^2.
\end{multline}
\end{corollary}
To pass from~\eqref{eq:permutation version} to~\eqref{eq:weighted quadruple} note that~\eqref{eq:weighted quadruple} is homogeneous of order $2$ in $(p_1,p_2,p_3,p_4)$, so we may assume that $p_1+p_2+p_3+p_4=1$. Now~\eqref{eq:weighted quadruple} is a direct application of~\eqref{eq:permutation version} with the above specific choice of permutation $\pi$, while subtracting from both sides of~\eqref{eq:permutation version} those multiples of $d_X(x_1,x_3)^2$ and $d_X(x_2,x_4)^2$  that appear in the right hand side of~\eqref{eq:permutation version}.

When $p_1+p_3=p_2+p_4=1$, Corollary~\ref{coro:three degrees of freedom} becomes Sturm's weighted quadruple inequality~\cite{Stu03}, which asserts that for every Hadamard space $(X,d_X)$, every $x_1,x_2,x_3,x_4\in X$, and every $s,t\in [0,1]$,
\begin{align}\label{eq:sturm weighted quadruple}
\nonumber &s(1-s)d_X(x_1,x_3)^2+t(1-t)d_X(x_2,x_4)^2\\  &\le std_X(x_1,x_2)^2+(1-s)td_X(x_2,x_3)^2 +(1-s)(1-t)d_X(x_3,x_4)^2+s(1-t)d_X(x_4,x_1)^2.
\end{align}
As explained in~\cite[Proposition~2.4]{Stu03}, by choosing the parameters $s,t$ appropriately in~\eqref{eq:weighted quadruple} one obtains an important quadruple comparison inequality of Reshetnyak~\cite{Res68} (see also~\cite{Jos97} or~\cite[Lemma~2.1]{LPS00}), asserting that for every Hadamard space $(X,d_X)$ and every $x_1,x_2,x_3,x_4\in X$,
\begin{equation}\label{eq:Reshetnyak}
d_X(x_1,x_3)^2+d_X(x_2,x_4)^2 \le d_X(x_1,x_2)^2+d_X(x_2,x_3)^2+2d_X(x_3,x_4)d_X(x_4,x_1).
\end{equation}

The coefficients in~\eqref{eq:weighted quadruple} have $3$ degrees of freedom while in~\eqref{eq:sturm weighted quadruple} they have $2$ degrees of freedom. This additional flexibility yields a proof of the validity of the Ptolemy inequality~\eqref{eq:potelmy intro} in Hadamard spaces. The fact that the Ptolemy inequality holds true in Hadamard spaces was proved in~\cite{FLS07}, and an alternative proof was given in~\cite{BFW09}. Both of these proofs rely on comparisons with ideal configurations in the Euclidean plane (see~\cite[\S II.1]{BH99}), combined with the classical Ptolemy theorem in Euclidean geometry. Corollary~\ref{coro:potlemy} below shows how the Ptolemy inequality is a direct consequence of~\eqref{eq:weighted quadruple}, thus yielding an intrinsic proof that does not proceed through an embedding argument.

\begin{corollary}\label{coro:potlemy}
Let $(X,d_X)$ be a Hadamard space and $x_1,x_2,x_3,x_4\in X$. Write $d_{ij}=d_X(x_i,x_j)$ for every $i,j\in \n$. Then
\begin{equation}\label{ptolmey with defect}
d_{12}d_{34}+d_{23}d_{41}-d_{13}d_{24}
\ge \frac{\big(\left(d_{12}d_{23}+d_{34}d_{41}\right)d_{13}-\left(d_{12}d_{41}+d_{23}d_{34}\right)d_{24}\big)^2}{2\left(d_{12}d_{41}+d_{23}d_{34}\right)
\left(d_{12}d_{23}+d_{34}d_{41}\right)}\ge 0.
\end{equation}
\end{corollary}

\begin{proof} The proof of~\eqref{ptolmey with defect} is nothing more than an application of Corollary~\ref{coro:three degrees of freedom} with the following specific choices of $p_1,p_2,p_3,p_4\in [0,\infty)$.
$$
p_1\eqdef \frac{d_{34}}{d_{41}}\cdot \frac{d_{23}+d_{41}}{d_{12}+d_{34}},\quad p_2\eqdef \frac{d_{41}}{d_{12}}\cdot \frac{d_{12}+d_{34}}{d_{23}+d_{41}},\quad p_3\eqdef \frac{d_{12}}{d_{23}}\cdot \frac{d_{23}+d_{41}}{d_{12}+d_{34}},\quad p_4\eqdef \frac{d_{23}}{d_{34}}\cdot \frac{d_{12}+d_{34}}{d_{23}+d_{41}}.
$$
A substitution of these values into~\eqref{eq:weighted quadruple} yields
\begin{align*}
2d_{12}d_{34}+2d_{23}d_{41}&\ge \frac{d_{12}d_{23}+d_{34}d_{41}}{d_{12}d_{41}+d_{23}d_{34}}d_{13}^2
+\frac{d_{12}d_{41}+d_{23}d_{34}}{d_{12}d_{23}+d_{34}d_{41}}d_{24}^2\\
&=2d_{13}d_{24}+\frac{\left(\left(d_{12}d_{23}+d_{34}d_{41}\right)d_{13}-\left(d_{12}d_{41}+d_{23}d_{34}\right)d_{24}\right)^2}{\left(d_{12}d_{41}+d_{23}d_{34}\right)
\left(d_{12}d_{23}+d_{34}d_{41}\right)}.\tag*{\qedhere}
\end{align*}
\end{proof}

\subsection{Iterative applications of Lemma~\ref{lem:harmonic}}\label{sec:iteration}  The case $s=t=1/2$ of~\eqref{eq:weighted quadruple} becomes the roundness $2$ inequality~\eqref{eq:roundness 2}, i.e., for every Hadamard space $(X,d_X)$ and every $x_1,x_2,x_3,x_4\in X$ we have
\begin{equation}\label{eq:roundness 2 in section}
d_X(x_1,x_3)^2+d_X(x_2,x_4)^2\le d_X(x_1,x_2)^2+d_X(x_2,x_3)^2+d_X(x_3,x_4)^2+d_X(x_4,x_1)^2.
\end{equation}
In~\cite{Enf76}, Enflo iterated~\eqref{eq:roundness 2 in section} (while exploiting cancellations) so as to yield the following inequality, which holds for every Hadamard space $(X,d_X)$, every $n\in \N$ and every $f:\{-1,1\}^n\to X$.
\begin{equation}\label{eq:enf type}
\sum_{x\in \{-1,1\}^n} d_X\left(f(x),f(-x)\right)^2\le \sum_{i=1}^n\sum_{x\in \{-1,1\}^n} d_X\left(f(x),f(x_1,\ldots,x_{i-1},-x_i,x_{i+1},\ldots,x_n)\right)^2.
\end{equation}
In today's terminology~\eqref{eq:enf type} says that every Hadamard space has Enflo type $2$ with constant $1$ (see also~\cite{OP09}). The argument in~\cite{LN04} yields a different iterative application of~\eqref{eq:roundness 2 in section} (again, exploiting cancellations via a telescoping argument), showing that mappings from the iterated diamond graph (see~\cite{NR03}) into any Hadamard space satisfy a certain quadratic metric inequality. Similar reasoning (as in~\cite{LMN05}) yields a quadratic metric inequality for  Hadamard space-valued mappings on the Laakso graphs (see~\cite{Laa02,LP01}). The value of the above iterative applications of~\eqref{eq:roundness 2 in section} is that they yield inequalities on metric spaces of unbounded cardinality (hypercubes, diamond graphs, Laakso graphs) that serve as obstructions to bi-Lipschitz embeddings of these spaces into any Hadamard space: these inequalities imply that any such embedding must incur distortion that tends to $\infty$ as the size of the underlying space tends to $\infty$ (in fact, these inequalities yield sharp bounds).

We therefore see that by applying Lemma~\ref{lem:harmonic} multiple times one could obtain quadratic metric inequalities that yield severe restrictions on those metric spaces that admit a bi-Lipschits embedding into some Hadamard space. Specifically, one could apply Lemma~\ref{lem:harmonic}  to several configurations of points and several choices of weights, and consider a weighted average of the resulting inequalities. This yields the estimate
 \begin{equation}\label{eq:summed version}
\sum_{i=1}^n\sum_{j=1}^n \sum_{k=1}^m \frac{c_ka_{ij}^kb_{ij}^k}{a_{ij}^k+b_{ij}^k}d_X(x_i,x_j)^2\le \sum_{i=1}^n\sum_{j=1}^n \sum_{k=1}^m c_k p_i^kq_j^k d_X(x_i,x_j)^2,
\end{equation}
 which is valid for every Hadamard space $(X,d_X)$, every $m,n\in \N$, every $x_1,\ldots,x_n\in X$, every $\{c_k\}_{k=1}^m\subset (0,\infty)$, every
 $\big\{p_i^k,q_i^k:\ i\in \n,\ k\in \{1,\ldots,m\}\big\}\subset (0,\infty)$ with
 \begin{equation}\label{eq:pik}
 \sum_{i=1}^np_i^k=\sum_{j=1}^nq_j^k=1,
 \end{equation}
  and every choice of $n$ by $n$ matrices  $\{A_k=(a_{ij}^k)\}_{k=1}^m,\{B_k=(b_{ij}^k)\}_{k=1}^m\subset M_n(\R)$ with nonnegative entries,  such that for every $i,j\in \n$ and $k\in \{1,\ldots,m\}$,
\begin{equation}\label{eq:AkBk}
\sum_{s=1}^n a_{is}^k+\sum_{s=1}^n b_{sj}^k=p_i^k+q_j^k.
\end{equation}

By collecting terms in~\eqref{eq:summed version} so that for every $i,j\in \n$ no multiple of $d_X(x_i,x_j)^2$ appears in both sides of the inequality, as was done in e.g.~\eqref{eq:enf type}, one arrives at the following estimate.
\begin{multline}\label{eq:collect terms2}
\sum_{\substack{i,j\in \n\\ \sum_{k=1}^mc_k\big(\frac{a_{ij}^kb_{ij}^k}{a_{ij}^k+b_{ij}^k}- p_i^kq_j^k\big)>0}}
\sum_{k=1}^mc_k\bigg(\frac{a_{ij}^kb_{ij}^k}{a_{ij}^k+b_{ij}^k}- p_i^kq_j^k\bigg) d_X(x_i,x_j)^2\\\le
\sum_{\substack{i,j\in \n\\ \sum_{k=1}^mc_k\big(\frac{a_{ij}^kb_{ij}^k}{a_{ij}^k+b_{ij}^k}- p_i^kq_j^k\big)<0}}
\sum_{k=1}^m c_k\bigg(p_i^kq_j^k-\frac{a_{ij}^kb_{ij}^k}{a_{ij}^k+b_{ij}^k}\bigg) d_X(x_i,x_j)^2.
\end{multline}

To the best of our knowledge, all of the previously used quadratic metric inequalities on general Hadamard spaces are of the form~\eqref{eq:collect terms2}. We therefore ask whether the inequalities of the form~\eqref{eq:collect terms2} capture the totality of those quadratic metric inequalities that are valid in Hadamard spaces.

\begin{question}\label{eq:is this all?}
Is it true that for every $D\in [1,\infty)$ there exists some $c(D)\in [1,\infty)$ such that a metric space $(X,d_X)$ embeds with distortion at most $c(D)$ into some Hadamard space provided
\begin{multline*}
\sum_{\substack{i,j\in \n\\ \sum_{k=1}^mc_k\big(\frac{a_{ij}^kb_{ij}^k}{a_{ij}^k+b_{ij}^k}- p_i^kq_j^k\big)>0}}
\sum_{k=1}^mc_k\bigg(\frac{a_{ij}^kb_{ij}^k}{a_{ij}^k+b_{ij}^k}- p_i^kq_j^k\bigg) d_X(x_i,x_j)^2\\\le D^2\cdot
\sum_{\substack{i,j\in \n\\ \sum_{k=1}^mc_k\big(\frac{a_{ij}^kb_{ij}^k}{a_{ij}^k+b_{ij}^k}- p_i^kq_j^k\big)<0}}
\sum_{k=1}^m c_k\bigg(p_i^kq_j^k-\frac{a_{ij}^kb_{ij}^k}{a_{ij}^k+b_{ij}^k}\bigg) d_X(x_i,x_j)^2,
\end{multline*}
for all $m,n\in \N$,  all $c_k, p_i^k, q_i^k, a_{ij}^k, b_{ij}^k\in [0,\infty)$ satisfying~\eqref{eq:pik} and~\eqref{eq:AkBk}, and all $x_1,\ldots,x_n\in X$?
\end{question}

Recall that there are useful metric inequalities, which are not quadratic metric inequalities, that hold true in any Hadamard space, such as Reshetnyak's inequality~\eqref{eq:Reshetnyak} or the Ptolemy inequality~\eqref{eq:potelmy intro}. However, we already know through Proposition~\ref{prop:duality} that quadratic metric inequalities fully characterize subsets of Hadamard spaces. And, in the case of Reshetnyak's inequality or the   Ptolemy inequality, we have seen above how to deduce them explicitly from a quadratic metric inequality (the key point to note here is that the various coefficients that appear in~\eqref{eq:collect terms2} can be optimized so as to depend on the distances $\{d_X(x_i,x_j)\}_{i,j\in \n}$).


A negative answer to  Question~\ref{eq:is this all?} would be very interesting, as it would yield a new family of metric spaces that fail to admit a bi-Lipschitz embedding into any Hadamard space, and correspondingly a new family of quadratic metric inequalities which hold true in any Hadamard space yet do not follow from the above procedure for obtaining such inequalities.

As discussed in the Introduction, it is not known whether or not  for every metric space $(X,d_X)$ there exists a Hadamard space $(Y,d_Y)$ with $c_Y(X,\sqrt{d_X})<\infty$. If this were true then Question~\ref{eq:sqrt of metric in hadamard} below would have a positive answer. Conversely, a positive answer to both Question~\ref{eq:is this all?} and Question~\ref{eq:sqrt of metric in hadamard} would imply that the 1/2-snowflake of any metric space admits a bi-Lipschitz embedding into some Hadamard space.

\begin{question}\label{eq:sqrt of metric in hadamard}
Is it true that every metric space $(X,d_X)$ satisfies
\begin{multline*}
\sum_{\substack{i,j\in \n\\ \sum_{k=1}^mc_k\big(\frac{a_{ij}^kb_{ij}^k}{a_{ij}^k+b_{ij}^k}- p_i^kq_j^k\big)>0}}
\sum_{k=1}^mc_k\bigg(\frac{a_{ij}^kb_{ij}^k}{a_{ij}^k+b_{ij}^k}- p_i^kq_j^k\bigg) d_X(x_i,x_j)\\\lesssim
\sum_{\substack{i,j\in \n\\ \sum_{k=1}^mc_k\big(\frac{a_{ij}^kb_{ij}^k}{a_{ij}^k+b_{ij}^k}- p_i^kq_j^k\big)<0}}
\sum_{k=1}^m c_k\bigg(p_i^kq_j^k-\frac{a_{ij}^kb_{ij}^k}{a_{ij}^k+b_{ij}^k}\bigg) d_X(x_i,x_j),
\end{multline*}
for all $m,n\in \N$,  all $c_k, p_i^k, q_i^k, a_{ij}^k, b_{ij}^k\in [0,\infty)$ satisfying~\eqref{eq:pik} and~\eqref{eq:AkBk}, and all $x_1,\ldots,x_n\in X$?
\end{question}
Question~\ref{eq:sqrt of metric in hadamard} seems tractable, but at present we do not know whether or not its answer is positive. A negative answer to Question~\ref{eq:sqrt of metric in hadamard} would yield for the first time a metric space $(X,d_X)$ such that $(X,\sqrt{d_X})$ fails to admit a bi-Lipschitz embedding into any Hadamard space, in sharp contrast to the case of embeddings into Alexandrov spaces of nonnegative curvature. In the same vein, a proof that every Hadamard space admits a sequence of bounded degree expanders would resolve Question~\ref{eq:sqrt of metric in hadamard} negatively. It is true that inequality~\eqref{eq:weighted quadruple} is not an obstruction to the validity of Question~\ref{eq:sqrt of metric in hadamard}, i.e., for every metric space $(X,d_X)$, every $p_1,p_2,p_3,p_4\in [0,\infty)$ and every $x_1,x_2,x_3,x_4\in X$ we have
\begin{multline}\label{eq:weighted quadruple no squares}
p_1p_2 d_X(x_1,x_2)+p_2p_3d_X(x_2,x_3)+p_3p_4d_X(x_3,x_4)+p_4p_1d_X(x_4,x_1)\\\ge
\frac{p_1p_3(p_2+p_4)}{p_1+p_3}d_X(x_1,x_3)+\frac{p_2p_4(p_1+p_3)}{p_2+p_4}d_X(x_2,x_4).
\end{multline}
Also (actually, as a consequence of~\eqref{eq:weighted quadruple no squares}), Reshetnyak's inequality and the   Ptolemy inequality hold true in any square root of a metric space, i.e., for every metric space $(X,d_X)$ and $x_1,x_2,x_3,x_4\in X$,
\begin{equation*}\label{eq:Reshetnyak root}
d_X(x_1,x_3)+d_X(x_2,x_4) \le d_X(x_1,x_2)+d_X(x_2,x_3)+2\sqrt{d_X(x_3,x_4)d_X(x_4,x_1)},
\end{equation*}
and
\begin{equation*}\label{eq:Ptolemy root}
\sqrt{d_X(x_1,x_3)d_X(x_2,x_4)}\le \sqrt{d_X(x_1,x_2)d_X(x_3,x_4)}+\sqrt{d_X(x_2,x_3)d_X(x_4,x_1)}.
\end{equation*}
It is possible (and instructive) to prove these inequalities while using only the triangle inequality, but this seems to require a somewhat tedious case analysis. Alternatively, one could verify~\eqref{eq:weighted quadruple no squares} by using the  fact that the square root of any four-point metric space admits an isometric embedding into a Hilbert space; see e.g.~\cite[Proposition~2.6.2]{CF94}.

Lemma~\ref{lem:factor 3} below asserts that the conclusion of Lemma~\ref{lem:harmonic} holds true in any square-root of a metric space, with a loss of a constant factor. This is a special case of Question~\ref{eq:sqrt of metric in hadamard} that falls sort of a positive answer in general due to the fact that we want to iterate the resulting inequality, in which case the constant factor loss could accumulate.

\begin{lemma}\label{lem:factor 3}
Fix $n\in \N$ and $p_1,\ldots,p_n,q_1,\ldots,q_n\in (0,1)$ such that $\sum_{i=1}^np_i=\sum_{j=1}^nq_j=1$. Suppose that $A=(a_{ij}), B=(b_{ij})\in M_n(\R)$ are $n$ by $n$ matrices with nonnegative entries that satisfy~\eqref{eq:AB assumption}. Then for every metric space $(X,d_X)$ and every $x_1,\ldots,x_n\in X$ we have
\begin{equation*}
\sum_{i=1}^n\sum_{j=1}^n \frac{a_{ij}b_{ij}}{a_{ij}+b_{ij}}d_X(x_i,x_j)\le 3\sum_{i=1}^n\sum_{j=1}^n p_iq_j d_X(x_i,x_j).
\end{equation*}
\end{lemma}

\begin{proof} Let $F:\{x_1,\ldots,x_n\}\to \ell_\infty$ be any isometric embedding of the metric space $(\{x_1,\ldots,x_n\},d_X)$ into $\ell_\infty$. By convexity we have
\begin{multline*}
\sum_{i=1}^n p_i\Big\|F(x_i)-\sum_{j=1}^n q_j F(x_j)\Big\|_\infty=\sum_{i=1}^n p_i\Big\|\sum_{j=1}^n q_j(F(x_i)- F(x_j))\Big\|_\infty\\ \le \sum_{i=1}^n\sum_{j=1}^n p_iq_j\|F(x_i)-F(x_j)\|_\infty=\sum_{i=1}^n\sum_{j=1}^n p_iq_jd_X(x_i,x_j),
\end{multline*}
and similarly,
\begin{multline*}
\sum_{i=1}^n q_i\Big\|F(x_i)-\sum_{j=1}^n q_j F(x_j)\Big\|_\infty\le \sum_{i=1}^n q_i\Big\|F(x_i)-\sum_{j=1}^n p_j F(x_j)\Big\|_\infty+\Big\|\sum_{i=1}^n\sum_{j=1}^n p_iq_j(F(x_i)-F(x_j))\Big\|_\infty\\
\le 2\sum_{i=1}^n\sum_{j=1}^n p_iq_j\|F(x_i)-F(x_j)\|_\infty
=2\sum_{i=1}^n\sum_{j=1}^n p_iq_jd_X(x_i,x_j).
\end{multline*}
So, if we denote $z\eqdef \sum_{k=1}^n q_k F(x_k)$ then
\begin{align*}
3\sum_{i=1}^n\sum_{j=1}^n p_iq_jd_X(x_i,x_j)&\ge \sum_{i=1}^n p_i\|F(x_i)-z\|_\infty+\sum_{i=1}^n q_j\|F(x_j)-z\|_\infty\\
&\!\stackrel{\eqref{eq:AB assumption}}{=}\sum_{i=1}^n\sum_{j=1}^n \left(a_{ij}\|F(x_i)-z\|_\infty+b_{ij}\|F(x_j)-z\|_\infty\right)\\
&\ge \sum_{i=1}^n\sum_{j=1}^n \min\{a_{ij},b_{ij}\}\left(\|F(x_i)-z\|_\infty+\|F(x_j)-z\|_\infty\right)\\
&\ge \sum_{i=1}^n\sum_{j=1}^n\frac{a_{ij}b_{ij}}{a_{ij}+b_{ij}}\|F(x_i)-F(x_j)\|_\infty\\
&= \sum_{i=1}^n\sum_{j=1}^n\frac{a_{ij}b_{ij}}{a_{ij}+b_{ij}}d_X(x_i,x_j).\qedhere
\end{align*}
\end{proof}
\subsection{A hierarchy of quadratic metric inequalities}\label{sec:hierarchy} The quadratic metric inequalities of Section~\ref{sec:iteration}  are part of a first level of a hierarchy of quadratic metric inequalities that hold true in any Hadamard space. We shall now describe these inequalities, which quickly become quite complicated and unwieldy.  We conjecture that the entire hierarchy of inequalities thus obtained characterizes subsets of Hadamard spaces; see Question~\ref{Q:hierarchy} below. Due to the generality of these inequalities, this conjecture could be quite tractable. But, even if it has a positive answer then it would yield a complicated, and therefore perhaps less useful, characterization of subsets of Hadamard spaces, and it would still be very interesting to find a smaller family of inequalities that characterizes subsets of Hadamard spaces, in the spirit of Question~\ref{eq:is this all?}.

Let $(X,d_X)$ be a Hadamard space. The barycentric inequality~\eqref{eq:B axiom} has the following counterpart as a formal consequence, which is an inequality that allows one to control the distance between barycenters of two probability measures. Let $\mu,\nu$ be finitely supported probability measures on $X$. By applying~\eqref{eq:B axiom} twice we see that
\begin{multline*}
d_X(\B(\nu),\mathfrak{B}(\mu))^2+\int_X d_X(\mathfrak{B}(\mu),x)^2\d \mu(x)\le \int_X d_X(\B(\nu),x)^2\d \mu(x)\\
\le \int_X \left(\int_X d_X(x,y)\d \nu(y)-\int_X d_X(\B(\nu),y)^2\d \nu(y)\right)\d \mu(x).
\end{multline*}
Thus
\begin{multline}\label{eq:two centers}
d_X(\B(\nu),\mathfrak{B}(\mu))^2+\int_X d_X(\mathfrak{B}(\mu),x)^2\d \mu(x)+\int_X d_X(\B(\nu),y)^2\d \nu(y)\\\le \iint_{X\times X} d_X(x,y)^2\d \mu(x)\d \nu(y).
\end{multline}
Both~\eqref{eq:B axiom} and~\eqref{eq:two centers} will be used repeatedly in what follows.

\subsubsection{An inductive construction}\label{sec:iterated barycenters} Fix $n\in \N$ and $x_1,\ldots,x_n\in X$. Fix also a sequence of integers $\{m_s\}_{s=0}^\infty\subset \N$ with $m_0=n$. Suppose that we are given $\mu^{k+1,b}_{s,a}\in [0,\infty)$ for every $k\in \N\cup\{0\}$,  $s\in \{0,\ldots,k\}$, $a\in \{1,\ldots, m_s\}$ and $b\in \{1,\ldots,m_{k+1}\}$, such that
 \begin{equation*}
\forall\, b\in \{1,\ldots,m_{k+1}\}, \qquad \sum_{s=0}^{k}\sum_{a=1}^{m_s} \mu^{k+1,b}_{s,a}=1.
\end{equation*}

We shall now proceed to define by induction on $k\in \N\cup \{0\}$ auxiliary points $x_a^s\in X$ for every $s\in \{0,\ldots,k\}$ and $a\in \{1,\ldots,m_s\}$. Our construction will also yield for every $i,j\in \n$, $s,t,\sigma,\tau\in \{0,\ldots,k\}$, $a\in \{1,\ldots,m_s\}$, $\alpha\in \{1,\ldots,m_\sigma\}$, $b\in \{1,\ldots,m_t\}$ and $\beta\in \{1,\ldots,m_\tau\}$ nonnegative weights $U^{s,t,a,b}_{\sigma,\tau,\alpha,\beta}, V_{i,j}^{s,t,a,b}\in [0,\infty)$  that satisfy the inequality
\begin{equation}\label{eq:UV ineq}
d_X(x_a^s,x_b^t)^2+\sum_{\sigma=0}^k\sum_{\tau=0}^k\sum_{\alpha=1}^{m_\sigma}\sum_{\beta=1}^{m_\tau} U^{s,t,a,b}_{\sigma,\tau,\alpha,\beta}d_X(x_\alpha^\sigma,x_\beta^\tau)^2 \le \sum_{i=1}^n\sum_{j=1}^n V_{i,j}^{s,t,a,b} d_X(x_i,x_j)^2.
\end{equation}

The induction starts by setting $x^0_a=x_a$ for $a\in \n$. Also, for every $a,b,\alpha,\beta\in \n$  set  $U^{0,0,a,b}_{0,0,\alpha,\beta}=0$ and $V_{\alpha,\beta}^{0,0,a,b}=\1_{\{(\alpha,\beta)=(a,b)\}}$, thus satisfying~\eqref{eq:UV ineq} vacuously.

Suppose now that we have defined $x_a^s\in X$ for every $s\in \{0,\ldots,k\}$ and $a\in \{1,\ldots,m_s\}$. Consider the probability measures
$$
\forall\, b\in \{1,\ldots,m_{k+1}\}, \qquad \mu^{k+1,b}\eqdef \sum_{s=0}^{k}\sum_{a=1}^{m_s} \mu^{k+1,b}_{s,a} \delta_{x_a^s},
$$
and define
$$
\forall\, b\in \{1,\ldots,m_{k+1}\}, \qquad x_b^{k+1}\eqdef \B\left(\mu^{k+1,b}\right).
$$

 Suppose that $s\in \{0,\ldots,k\}$, $a\in \{1,\ldots,m_s\}$ and $b\in \{1,\ldots,m_{k+1}\}$. Then by~\eqref{eq:B axiom} we have
$$
d_X(x_a^s,x_b^{k+1})^2+\sum_{\tau=0}^{k}\sum_{\beta=1}^{m_\tau} \mu^{k+1,b}_{\tau,\beta} d_X(x_b^{k+1},x_\beta^\tau)^2\le \sum_{t=0}^{k}\sum_{c=1}^{m_t} \mu^{k+1,b}_{t,c} d_X(x_a^s,x_c^t)^2.
$$
In combination with the inductive hypothesis~\eqref{eq:UV ineq}, this implies that the desired estimate~\eqref{eq:UV ineq} would also hold true when $|\{s,t\}\cap {k+1}|=1$ once we introduce the following inductive definitions.
 \begin{align*}
U^{k+1,s,b,a}_{\sigma,\tau,\beta,\alpha}&=U^{s,k+1,a,b}_{\sigma,\tau,\alpha,\beta}\\&\eqdef \1_{\{(\sigma,\alpha,\tau)\in \{k+1\}\times \{b\}\times \{0,\ldots,k\}\}}\mu^{k+1,b}_{\tau,\beta}
+\1_{\{\{\sigma,\tau\}\subset \{0,\ldots,k\}\}}\sum_{t=0}^{k}\sum_{c=1}^{m_t} \mu^{k+1,b}_{t,c}U^{s,t,a,c}_{\sigma,\tau,\alpha,\beta},
\end{align*}
and
$$
V^{s,k+1,a,b}_{i,j}\eqdef \sum_{t=0}^{k}\sum_{c=1}^{m_t} \mu^{k+1,b}_{t,c}V_{i,j}^{s,t,a,c}.
$$

It remains to ensure the validity of~\eqref{eq:UV ineq} when $s=t=k+1$. So, fix $a,b\in \{1,\ldots,k+1\}$ and apply~\eqref{eq:two centers} so as to obtain the estimate
\begin{align*}
d_X(x_a^{k+1},x_b^{k+1})^2&+\sum_{\tau=0}^{k}\sum_{\beta=1}^{m_\tau} \mu^{k+1,a}_{\tau,\beta} d_X(x_a^{k+1},x_\beta^\tau)^2+\sum_{\tau=0}^{k}\sum_{\beta=1}^{m_\tau} \mu^{k+1,b}_{\tau,\beta} d_X(x_b^{k+1},x_\beta^\tau)^2\\
&\le \sum_{t=0}^k\sum_{\theta=0}^k\sum_{p=1}^{m_t}\sum_{q=1}^{m_\theta} \mu^{k+1,a}_{t,p} \mu^{k+1,b}_{\theta,q}  d_X(x_p^t,x_q^\theta)^2.
\end{align*}
In combination with the inductive hypothesis~\eqref{eq:UV ineq}, this implies that the desired estimate~\eqref{eq:UV ineq} would also hold true when $s=t=k+1$ once we introduce the following inductive definitions.
\begin{align*}
&U^{k+1,k+1,a,b}_{\sigma,\tau,\alpha,\beta}\\&\eqdef \1_{\{(\sigma,\alpha,\tau)\in \{k+1\}\times\{a,b\}\times \{0,\ldots,k\}\}}\mu_{\tau,\beta}^{k+1,\alpha}+ \1_{\{\{\sigma,\tau\}\subset \{0,\ldots,k\}\}}\sum_{t=0}^k\sum_{\theta=0}^k\sum_{p=1}^{m_t}\sum_{q=1}^{m_\theta} \mu^{k+1,a}_{t,p} \mu^{k+1,b}_{\theta,q}U^{t,\theta,p,q}_{\sigma,\tau,\alpha,\beta},
\end{align*}
and
$$
V_{i,j}^{k+1,k+1,a,b}=\sum_{t=0}^k\sum_{\theta=0}^k\sum_{p=1}^{m_t}\sum_{q=1}^{m_\theta} \mu^{k+1,a}_{t,p} \mu^{k+1,b}_{\theta,q}V_{i,j}^{t,\theta,p,q}.
$$

This concludes our inductive construction of auxiliary points, which satisfy the inequality~\eqref{eq:UV ineq}. We shall now show how to remove the auxiliary points so as to obtain bona fide quadratic metric inequalities that involve only points from the subset $\{x_1,\ldots,x_n\}\subset X$.

\subsubsection{Deriving quadratic metric inequalities}\label{sec:deriving} Suppose that for every $s,t\in \{0,\ldots, k\}$, $a\in \{0,\ldots,m_s\}$ and $b\in \{0,\ldots,m_t\}$ we are given a nonnegative weight $\Gamma^{s,t}_{a,b}\in [0,\infty)$. By multiplying~\eqref{eq:UV ineq} by $\Gamma^{s,t}_{a,b}$ and summing the resulting inequalities, we obtain the estimate
\begin{equation}\label{eq:EF summed}
\sum_{s=0}^k\sum_{t=0}^k\sum_{a=1}^{m_s}\sum_{b=1}^{m_t} E^{s,t}_{a,b} d_X(x_a^s,x_b^t)^2\le \sum_{i=1}^n\sum_{j=1}^n F_{i,j} d_X(x_i,x_j)^2,
\end{equation}
where
$$
E^{s,t}_{a,b}\eqdef \Gamma^{s,t}_{a,b}+\sum_{\sigma=0}^k\sum_{\tau=0}^k\sum_{\alpha=1}^{m_\sigma}\sum_{\beta=1}^{m_\tau} \Gamma^{\sigma,\tau}_{\alpha,\beta}U^{\sigma,\tau,\alpha,\beta}_{s,t,a,b},
$$
and
$$
F_{i,j}\eqdef \sum_{s=0}^k\sum_{t=0}^k\sum_{a=1}^{m_s}\sum_{b=1}^{m_t} \Gamma^{s,t}_{a,b} V_{ij}^{s,t,a,b}.
$$

Denote
$$\mathsf{S}_k\eqdef \big\{x_a^s:\ s\in \{0,\ldots,k\}\ \mathrm{and}\ a\in \{1,\ldots, m_s\}\big\}\subset X.$$
Any $\zeta\in \bigcup_{\ell=1}^\infty \mathsf{S}_k^\ell$ will be called below a path in $\mathsf{S}_k$. If  $\zeta=(\zeta_0,\ldots,\zeta_\ell)$ for some $\ell\in \N$ then we  write $\ell(\zeta)=\ell$.  The points $\zeta_0, \zeta_{\ell(\zeta)}$ are called the endpoints of the path $\zeta$. The path $\zeta$ is called non-repetitive if the points $\zeta_0,\ldots,\zeta_{\ell(\zeta)}$ are distinct. The finite set of all non-repetitive paths $\zeta$ in $\mathsf{S}_k$ whose endpoints satisfy $\{\zeta_0,\zeta_{\ell(\zeta)}\}\subset \{x_1,\ldots,x_n\}$ will be denoted below by $\mathsf{P}_{\!k}$. Suppose that for every path $\zeta\in \mathsf{P}_{\!k}$ we are given $c_1(\zeta),\ldots,c_{\ell(\zeta)}(\zeta)\in (0,\infty)$ such that for every $s,t\in \{0,\ldots,k\}$, $a\in \{1,\ldots,m_s\}$ and $b\in \{1,\ldots, m_t\}$ we have
$$
\sum_{\zeta\in \mathsf{P}_{\!k}} \sum_{r=1}^{\ell(\zeta)}c_r(\zeta)\1_{\{(\zeta_{r-1},\zeta_r)=(x_a^s,x_b^t)\}}=E^{s,t}_{a,b}.
$$
Then the inequality~\eqref{eq:EF summed} can be rewritten as follows.
\begin{equation}\label{eq:redistributed path}
\sum_{\zeta\in \mathsf{P}_{\!k}}\sum_{r=1}^{\ell(\zeta)} c_r(\zeta) d_X(\zeta_{r-1},\zeta_r)^2\le \sum_{i=1}^n\sum_{j=1}^n F_{i,j} d_X(x_i,x_j)^2.
\end{equation}

By the triangle inequality  and Cauchy–-Schwarz, every $\zeta\in \mathsf{P}_{\!k}$ satisfies
\begin{equation}
d_X(\zeta_0,\zeta_{\ell(\zeta)})^2\le\bigg(\sum_{r=1}^{\ell(\zeta)}\frac{1}{\sqrt{c_r(\zeta)}}\cdot \sqrt{c_r(\zeta)}d_X(\zeta_{r-1},\zeta_r)\bigg)^2
\le \bigg(\sum_{r=1}^{\ell(\zeta)}\frac{1}{c_r(\zeta)}\bigg)\sum_{r=1}^{\ell(\zeta)}c_r(\zeta)d_X(\zeta_{r-1},\zeta_r)^2.\label{eq:CS cs}
\end{equation}
By combining~\eqref{eq:redistributed path} and~\eqref{eq:CS cs} we therefore see that
\begin{equation}\label{eq:harmonic mean}
\sum_{\zeta\in \mathsf{P}_{\!k}}\frac{d_X(\zeta_0,\zeta_{\ell(\zeta)})^2}{\sum_{r=1}^{\ell(\zeta)}\frac{1}{c_r(\zeta)}}\le \sum_{i=1}^n\sum_{j=1}^n F_{i,j} d_X(x_i,x_j)^2.
\end{equation}
Recall that by the definition of $\mathsf{P}_{\!k}$,  the endpoints $\zeta_0,\zeta_{\ell(\zeta)}$ of any path $\zeta\in \mathsf{P}_{\!k}$ are in $\{x_1,\ldots,x_n\}$. It therefore follows from~\eqref{eq:harmonic mean} that if we define for every $i,j\in \n$
\begin{equation}\label{eq:GF}
G_{i,j}\eqdef \sum_{\substack{\zeta\in \mathsf{P}_{\!k}\\ (\zeta_0,\zeta_{\ell(\zeta)})=(x_i,x_j)}}\frac{1}{\sum_{r=1}^{\ell(\zeta)}\frac{1}{c_r(\zeta)}},
\end{equation}
then the following quadratic metric inequality, which generalizes~\eqref{eq:collect terms2}, holds true in every Hadamard space $(X,d_X)$.
\begin{equation*}
\sum_{\substack{i,j\in \n \\ G_{i,j}>F_{i,j}}} (G_{i,j}-F_{i,j}) d_X(x_i,x_j)^2\le \sum_{\substack{i,j\in \n\\ F_{i,j}>G_{i,j}}} (F_{i,j}-G_{i,j}) d_X(x_i,x_j)^2.
\end{equation*}

 \begin{question}\label{Q:hierarchy}
Is it true that for every $D\in [1,\infty)$ there exists some $C(D)\in [1,\infty)$ such that a metric space $(X,d_X)$ embeds with distortion at most $C(D)$ into some Hadamard space provided \begin{equation*}
\sum_{\substack{i,j\in \n \\ G_{i,j}>F_{i,j}}} (G_{i,j}-F_{i,j}) d_X(x_i,x_j)^2\le D^2\sum_{\substack{i,j\in \n\\ F_{i,j}>G_{i,j}}} (F_{i,j}-G_{i,j}) d_X(x_i,x_j)^2,
\end{equation*}
for every $n\in \N$, every $x_1,\ldots,x_n\in X$ and every $\{F_{i,j},G_{i,j}\}_{i,j\in \n}$ as in~\eqref{eq:GF}? Here we are considering all those $\{F_{i,j},G_{i,j}\}_{i,j\in \n}$ that  are obtained from the construction that is described in Section~\ref{sec:iterated barycenters}  and Section~\ref{sec:deriving}, i.e., ranging over all the possible choices of weights $\mu^{k+1,b}_{s,a}, \Gamma^{s,t}_{a,b}, c_r(\zeta)$ that were introduced in the course of this construction.
 \end{question}
We conjecture that the answer to Question~\ref{Q:hierarchy} is positive. It may even be the case that one could take $C(D)=D$ in Question~\ref{Q:hierarchy}. A negative answer here would be of great interest, since it would require finding a family of quadratic metric inequalities that does not follow (even up to a constant factor) from the above hierarchy of inequalities.

\section{Remarks on Question~\ref{Q:dichotomy sharp wasserstein}}\label{sec:dichotomy question remarks}

Focusing for concreteness on the case $p=2$ of Question~\ref{Q:dichotomy sharp wasserstein}, recall that we are asking whether every $n$-point metric space $(X,d_X)$ satisfies
\begin{equation}\label{eq:goal sqrt dich}
c_{\mathscr{P}_{\!2}(\R^3)}(X)\lesssim \sqrt{\log n}.
\end{equation}

The conclusion of Theorem~\ref{thm:main snowflake}, i.e., the fact that the $1/2$-snowflake of every finite metric space embeds with $O(1)$ distortion into $\mathscr{P}_{\!2}(\R^3)$, does not on its own imply~\eqref{eq:goal sqrt dich}.  Indeed, let $\sqrt{\ell_\infty}$ denote the $1/2$-snowflake of $\ell_\infty$. Then the $1/2$-snowflake of every finite metric space embeds isometrically into $\sqrt{\ell_\infty}$. However, it is standard to check that if for $n\in \N$ we let $P_n$ denote the set $\n\subset \R$, equipped with the metric inherited from $\R$, then $c_{\sqrt{\ell_\infty}}(P_n)\gtrsim \sqrt{n}$. Thus, despite the fact that $\sqrt{\ell_\infty}$ is $1/2$-snowflake universal, the distortion of $n$-point metric spaces in $\sqrt{\ell_\infty}$ can grow much faster than the rate of $\sqrt{\log n}$ that we desire in~\eqref{eq:goal sqrt dich}. Nevertheless, $\sqrt{\ell_\infty}$ is not an especially convincing example in our context, since it does not contain rectifiable curves (which is essentially the reason for the lower bound $c_{\sqrt{\ell_\infty}}(P_n)\gtrsim \sqrt{n}$), while $\mathscr{P}_{\!2}(\R^3)$ is an Alexandrov space of nonnegative curvature.

Note that $c_{\mathscr{P}_{\!2}(\R^3)}(X)\lesssim \log  n$ for every $n$-point metric space $(X,d_X)$, so $\mathscr{P}_{\!2}(\R^3)$ certainly does not exhibit the bad behavior that we described above for embeddings into $\sqrt{\ell_\infty}$. This logarithmic upper bound follows from the fact that $c_{\mathscr{P}_{\!2}([0,1])}(X)\lesssim \log  n$, so in fact $c_{\Pp(Y)}(X)\lesssim \log n$ for  every metric space $(Y,d_Y)$ that contains a geodesic segment and every $n$-point metric space $(X,d_X)$. The bound $c_{\mathscr{P}_{\!2}([0,1])}(X)\lesssim \log  n$ is a consequence of Bourgain's embedding theorem~\cite{Bou85} combined with the easy fact that every finite subset of $\ell_2$ embeds with distortion $1$ into $\mathscr{P}_{\!2}([0,1])$. To check the latter assertion, take any $X\subset \ell_2$ of cardinality $n$. We may assume without loss of generality that $X\subset \R^n$. Denoting
$$
M\eqdef 1+\max_{x\in X}\max_{j\in \{1,\ldots,n-1\}} |x_{j+1}-x_j|,
$$
define $f:X\to \mathscr{P}_{\!2}(\R)$ by
$
f(x)\eqdef \sum_{j=1}^n \delta_{x_j+Mj}.
$
The choice of $M$ ensures that the sequence $\{x_j+Mj\}_{j=1}^n$ is strictly increasing, so for $x,y\in X$ the optimal transportation between $f(x)$ and $f(y)$ assigns the point mass at $x_j+Mj$ to the point mass at $y_j+Mj$ for every $j\in \n$. This shows that $\W_2(f(x),f(y))=\|x-y\|_2$. Since all the measures $\{f(x)\}_{x\in X}$ are supported on a bounded interval, by rescaling we obtain a distortion $1$ embedding of $X$ into  $\mathscr{P}_{\!2}([0,1])$.

An example that is more interesting in our context than $\sqrt{\ell_\infty}$, though still somewhat artificial, is the space $(\ell_2\oplus \sqrt{\ell_\infty})_2$. This space is $1/2$-snowflake universal (since it contains an isometric copy of $\sqrt{\ell_\infty}$) and also every $n$-point metric space $(X,d_X)$ satisfies $c_{(\ell_2\oplus \sqrt{\ell_\infty})_2}(X)\lesssim \log n$ (by Bourgain's theorem~\cite{Bou85}, since $(\ell_2\oplus \sqrt{\ell_\infty})_2$ contains an isometric copy of $\ell_2$). However, we shall prove below the following lemma which shows  that the conclusion of Question~\ref{Q:dichotomy sharp wasserstein} fails for $(\ell_2\oplus \sqrt{\ell_\infty})_2$.

\begin{lemma}\label{lem:sum of hilbert and sqrt l infty}
For arbitrarily large $n\in \N$ there exists an $n$-point metric space $(X_n,d_{X_n})$ that satisfies  $$c_{(\ell_2\oplus \sqrt{\ell_\infty})_2}(X_n)\gtrsim \log n.$$
\end{lemma}

Of course, $(\ell_2\oplus \sqrt{\ell_\infty})_2$ is still more pathological than $\mathscr{P}_{\!2}(\R^3)$ (in particular, not every pair of points in $(\ell_2\oplus \sqrt{\ell_\infty})_2$ can be joined by a rectifiable curve), and we lowered here the asymptotic growth rate of the largest possible distortion of an $n$-point metric space from the $O(\sqrt{n})$ of $\sqrt{\ell_\infty}$ to the $O(\log n)$ of $(\ell_2\oplus \sqrt{\ell_\infty})_2$ by artificially inserting a copy of $\ell_2$. Nevertheless, the proof of Lemma~\ref{lem:sum of hilbert and sqrt l infty} below illuminates the fact that in order to prove that Question~\ref{Q:dichotomy sharp wasserstein} has  a positive answer one would need to use properties of the Alexandrov space $\mathscr{P}_{\!2}(\R^3)$ that go beyond those that we isolated so far, and in particular it provides a concrete sequence of finite metric spaces for which the conclusion of Question~\ref{Q:dichotomy sharp wasserstein} is at present unknown; see Question~\ref{Q:subdivision} below.

Before proving Lemma~\ref{lem:sum of hilbert and sqrt l infty}, we set some notation. For a finite connected graph $G=(V_G,E_G)$, the shortest-path metric that $G$ induces on $V_G$ is denoted by $d_G$. For $k\in \N$, denote the $k$-fold subdivision of $G$ by $\Sigma_k(G)=(V_{\Sigma_k(G)},E_{\Sigma_k(G)})$, i.e., $\Sigma_k(G)$ is obtained from $G$ by replacing each edge $e\in E_G$ by a path consisting of $k$ edges joining the endpoints of $e$ (the interiors of these paths are disjoint for distinct $e,e'\in E_G$). Thus $|V_{\Sigma_k(G)}|=|V_G|+(k-1)|E_G|$. Note that the metric induced on $V_G\subset V_{\Sigma_k(G)}$ by the shortest-path metric $d_{\Sigma_k(G)}$ of $\Sigma_k(G)$ is a rescaling of $d_G$ by a factor of $k$, i.e.,
\begin{equation}\label{eq:rescale subdivision}
\forall\,x,y\in V_G\subset V_{\Sigma_k(G)},\qquad d_{\Sigma_k(G)}(x,y)=kd_G(x,y).
\end{equation}
 Suppose that $G$ is $d$-regular for some $d\in \N$. The normalized adjacency matrix of $G$, i.e., the $V_G\times V_G$ matrix whose entry at $u,v\in V_G$ equals $1/d$ if $\{u,v\}\in E_G$ and equals $0$ otherwise, is denoted $A_G$. The largest eigenvalue of the symmetric stochastic matrix $A_G$ equals $1$, and the second largest eigenvalue of $A_G$ is denoted $\lambda_2(G)$.

\begin{proof}[Proof of Lemma~\ref{lem:sum of hilbert and sqrt l infty}] Fix $d,n\in \N$. We shall show that if $G=(V_G,E_G)$ is an $n$-vertex $d$-regular graph then
\begin{equation}\label{eq:subdivision lower}
c_{(\ell_2\oplus \sqrt{\ell_\infty})_2}(\Sigma_k(G))\gtrsim  \min\left\{\sqrt{\frac{k\log n}{\log d}},\sqrt{1-\lambda_2(G)}\cdot \frac{\log n}{\log d}\right\}.
\end{equation}
In particular, for, say, $d=3$ and $\lambda_2(G)\le 99/100$, if $k\asymp \log n$ then
$$ c_{(\ell_2\oplus \sqrt{\ell_\infty})_2}(\Sigma_k(G))\asymp \log n\asymp \log |V_{\Sigma_k(G)}|.$$
This implies the validity of Lemma~\ref{lem:sum of hilbert and sqrt l infty} because arbitrarily large graphs with the above requirements are well-known to exist (see e.g.~\cite{HLW06}).

To prove~\eqref{eq:subdivision lower}, take $f:V_{\Sigma_k(G)}\to (\ell_2\oplus \sqrt{\ell_\infty})_2$ and suppose that there exist $s,D\in (0,\infty)$ such that for every  $x,y\in V_{\Sigma_k(G)}$ we have
$$
sd_{\Sigma_k(G)}(x,y)\le d_{(\ell_2\oplus \sqrt{\ell_\infty})_2}(f(x),f(y))\le Dsd_{\Sigma_k(G)}(x,y).
$$
Our goal is to bound $D$ from below. Writing $f(x)=(g(x),h(x))$ for every $x\in V_{\Sigma_k(G)}$, our assumption is that for every distinct $x,y\in V_{\Sigma_k(G)}$,
\begin{equation}\label{eq:sD direct sum}
1\le \frac{\|g(x)-g(y)\|_2^2+\|h(x)-h(y)\|_\infty}{s^2d_{\Sigma_k(G)}(x,y)^2} \le D^2.
\end{equation}

For $x,y\in V_{\Sigma_k(G)}$ with $\{x,y\}\in E_{\Sigma_k(G)}$  by~\eqref{eq:sD direct sum} we have $\|h(x)-h(y)\|_\infty\le s^2D^2=s^2D^2d_{\Sigma_k(G)}(x,y)$. Thus $h:V_{\Sigma_k(G)}\to \ell_\infty$ is $s^2D^2$-Lipschitz, and therefore
\begin{multline}\label{eq:use lipchitz on G}
\frac{1}{n^2}\sum_{x,y\in V_G}\|h(x)-h(y)\|_\infty\le \frac{s^2D^2}{n^2}\sum_{x,y\in V_G}d_{\Sigma_k(G)}(x,y)\\\stackrel{\eqref{eq:rescale subdivision}}{=}\frac{ks^2D^2}{n^2}\sum_{x,y\in V_G}d_{G}(x,y)\le
ks^2D^2\bigg(\frac{1}{n^2}\sum_{x,y\in V_G}d_{G}(x,y)^2\bigg)^{\frac12}.
\end{multline}
Consequently, by~\eqref{eq:sD direct sum} once more we have
\begin{align}\label{eq:l2 part dominant}
\frac{1}{n^2}\sum_{x,y\in V_G}\|g(x)-g(y)\|_2^2&\stackrel{\eqref{eq:sD direct sum}}{\ge} \nonumber \frac{s^2}{n^2}\sum_{x,y\in V_G}d_{\Sigma_k(G)}(x,y)^2-\frac{1}{n^2}\sum_{x,y\in V_G}\|h(x)-h(y)\|_\infty\\
&\stackrel{\eqref{eq:use lipchitz on G}}{\ge} \frac{k^2s^2}{n^2}\sum_{x,y\in V_G}d_G(x,y)^2-ks^2D^2\bigg(\frac{1}{n^2}\sum_{x,y\in V_G}d_{G}(x,y)^2\bigg)^{\frac12}.
\end{align}
At the same time, by the equivalent formulation of spectral gap in terms of a Poincar\'e inequality (see e.g.~\cite[Section~9.1]{Gro83} or~\cite{Mat97,MN14}),
\begin{multline}\label{eq:use poincare}
\frac{1}{n^2}\sum_{x,y\in V_G}\|g(x)-g(y)\|_2^2\le \frac{1}{1-\lambda_2(G)}\cdot \frac{2}{|E_G|}\sum_{\substack{x,y\in V_G\\ \{x,y\}\in E_G}}\|g(x)-g(y)\|_2^2\\
\stackrel{\eqref{eq:sD direct sum}}{\le} \frac{s^2D^2}{1-\lambda_2(G)}\cdot\frac{2}{|E_G|}\sum_{\substack{x,y\in V_G\\ \{x,y\}\in E_G}}d_{\Sigma_k(G)}(x,y)^2\stackrel{\eqref{eq:rescale subdivision}}{=} \frac{2s^2k^2D^2}{1-\lambda_2(G)}.
\end{multline}
By contrasting~\eqref{eq:l2 part dominant} with~\eqref{eq:use poincare} we deduce that
$$
D\gtrsim \min\left\{\bigg(\frac{1-\lambda_2(G)}{n^2}\sum_{x,y\in V_G}d_{G}(x,y)^2\bigg)^{\frac12},\bigg(\frac{k^2}{n^2}\sum_{x,y\in V_G}d_{G}(x,y)^2\bigg)^{\frac14}\right\}.
$$
This lower bound on $D$ implies the desired estimate~\eqref{eq:subdivision lower} since by a standard (and simple) counting argument (see e.g.~\cite[page~193]{Mat97}) the fact that $G$ has $n$ vertices and is $d$-regular implies that
\begin{equation*}
\bigg(\frac{1}{n^2}\sum_{x,y\in V_G}d_G(x,y)^2\bigg)^{\frac12}\gtrsim \frac{\log n}{\log d}. \qedhere
\end{equation*}
\end{proof}

\begin{question}\label{Q:subdivision}
Suppose that $G$ is an $n$-vertex $3$-regular graph with $\lambda_2(G)\le 99/100$. What is the asymptotic growth rate of
$$
c_{\mathscr{P}_{\!2}(\R^3)}\left(\Sigma_{\lceil \log n\rceil}(G)\right)?
$$
At present, the best known upper bound on this quantity is $O(\log n)$, while Question~\ref{Q:dichotomy sharp wasserstein} predicts that it is $O(\sqrt{\log n})$. Obtaining any $o(\log n)$ upper bound would be interesting here.
\end{question}


\bibliographystyle{abbrv}
\bibliography{w2}
\end{document}